\documentclass[11pt]{amsart}
\usepackage{dsfont}
\usepackage{amssymb, amsmath, amsthm}
\usepackage{mathrsfs}
\newcommand{\no}[1]{#1}
\renewcommand{\no}[1]{}
\no{\usepackage{times}\usepackage[subscriptcorrection, slantedGreek, nofontinfo]{mtpro}
\renewcommand{\Delta}{\upDelta}}
\usepackage{color}
\usepackage{graphicx}
\usepackage{enumerate,comment}

\usepackage{hyperref}
\usepackage{xcolor}
\hypersetup{
  colorlinks,
  linkcolor={blue!80!black},
  urlcolor={blue!80!black},
  citecolor={blue!80!black}
}

\newtheorem{theorem}{Theorem}[section]

\newtheorem{lem}{Lemma}[section]

\newtheorem{cor}{Corollary}[section]

\theoremstyle{remark}

\newcommand{\bel}{\begin{equation} \label}
\newcommand{\ee}{\end{equation}}

\newcommand{\R}{{\mathbb R}}

\def\phi {\varphi}

\renewcommand{\leq}{\leqslant}
\renewcommand{\geq}{\geqslant}

\def\beq{\begin{equation}}
\def\eeq{\end{equation}}
\newcommand{\bea}{\begin{eqnarray}}
\newcommand{\eea}{\end{eqnarray}}
\newcommand{\beas}{\begin{eqnarray*}}
\newcommand{\eeas}{\end{eqnarray*}}

\usepackage{multirow}
\usepackage{booktabs}
\newtheorem{example}{\bf{Example}}
\usepackage{subfigure}
\usepackage{epstopdf}
\usepackage{diagbox}
\providecommand{\abs}[1]{\left\lvert#1\right\rvert}
\providecommand{\norm}[1]{\left\lVert#1\right\rVert}

\numberwithin{equation}{section}

\title[]{Determination and reconstruction of a semilinear term from point measurements}

\author[Y. Kian]{Yavar Kian}
\address{Univ Rouen Normandie, CNRS, Normandie Univ, LMRS UMR 6085, F-76000 Rouen, France}
\email{yavar.kian@univ-rouen.fr}

\author[H. Liu]{Hongyu Liu}
\address{Department of Mathematics, City University of Hong Kong, Hong Kong SAR, China}
\email{hongyliu@cityu.edu.hk}

\author[L. Wang]{Li-Li Wang}
\address{School of Mathematics, Hunan University, Changsha 410082, China.}
\email{lilywang@hnu.edu.cn}

\author[G. Zheng]{Guang-Hui Zheng}
\address{School of Mathematics, Hunan University, Changsha 410082, China.}
\email{zhenggh2012@hnu.edu.cn}

\begin{document}
\begin{abstract}
In this article we study the inverse problem of determining a semilinear term appearing in an elliptic equation from boundary measurements. Our main objective is to develop flexible and general theoretical results that can be used for developing numerical reconstruction  algorithm for this inverse problem. For this purpose, we develop a new method, based on different properties of solutions of elliptic equations, for treating the determination of the semilinear term as a source term from a point measurement  of the solutions. This approach not only allows us to make important relaxations on the data used so far for solving this class of inverse problems, including general Dirichlet excitation lying in a space of dimension one and measurements located at one point on the boundary of the domain, but it also allows us to derive a novel algorithm for the reconstruction of the semilinear term. The effectiveness of our algorithm is corroborated by extensive numerical experiments. Notably, as demonstrated by the theoretical analysis, we are able to effectively reconstruct the unknown nonlinear source term by utilizing solely the information provided by the measurement data at a single point.


\end{abstract}


\maketitle


\section{Introduction}

Let $\Omega$ be a bounded and connected domain of $\mathbb{R}^n$, $n\geq 2$, with $C^{2+\alpha}$ boundary, $\alpha\in(0,1)$, boundary.  Let   $F\in C^1(\R)$ and let $a:=(a_{i,j})_{1 \leq i,j \leq n} \in C^2(\overline{\Omega};\R^{n^2})$
be symmetric, that is
$$
a_{i,j}(x)=a_{j,i}(x),\ x \in \Omega,\ i,j = 1,\ldots,n,
$$
and let $a$ fulfill the ellipticity condition: there exists a constant
$c>0$ such that
\bel{ell}
\sum_{i,j=1}^n a_{i,j}(x) \xi_i \xi_j \geq c |\xi|^2, \quad
\mbox{for each $x \in \overline{\Omega},\ \xi=(\xi_1,\ldots,\xi_n) \in \R^n$}.
\ee

We consider the following boundary value problem
\bel{eq1}
\left\{
\begin{array}{ll}
-|a|^{\frac{1}{2}}\sum_{i,j=1}^n \partial_{x_i}
\left( |a|^{-\frac{1}{2}}a_{i,j}(x) \partial_{x_j} u_\delta(x) \right)=F(u_\delta(x)) & \mbox{in}\ \Omega ,
\\
u_\delta=\delta g &\mbox{on}\ \partial\Omega,
\end{array}
\right.
\ee
with $|a|$  the absolute of value of the determinant of the matrix $a$ and  $g\in C^{2+\alpha}(\partial\Omega)$. In this article, we assume that the data $g$ is chosen in such way that $g(\partial\Omega)=[0,R]$, with $R>0$, and $\delta\in[0,1]$. We assume here  that $F$ is a non-increasing function and $F(0)=0$.

It is well known that \eqref{eq1}  admits a unique  solution $u_\delta\in C^{2+\alpha}(\overline{\Omega})$ (see \cite[Theorem 8.3, pp. 301]{LU} for the existence and \cite[Theorem 10.7]{GT} for the uniqueness), $\delta\in[0,1]$. Fixing $x_0\in\partial\Omega$, we consider in this article the determination of $F$ on $[0,R]$ from the knowledge of $\partial_{\nu_a} u_\delta(x_0)$, $\delta\in[0,1]$, with  $\nu(x)=(\nu_1(x),\ldots,\nu_n(x))$ the outward unit normal to $\partial\Omega$ computed at $x \in \partial\Omega$ and
$$\partial_{\nu_a} v(x)=\sum_{i,j=1}^na_{i,j}(x)\partial_{x_j}v(x)\nu_i(x),\quad x\in\partial\Omega.$$
More precisely, we study the unique and stable theoretical determination of such class of semilinear terms from this class of data as well as the corresponding numerical reconstruction.

Let us recall that semilinear elliptic equations of the form \eqref{eq1} can model different physical phenomenon. This includes problems of  spreading of biological populations or problems appearing in combustion theory associated with stationary solutions of reaction diffusion equations \cite{Z}. This class of equations appear also naturally in many models in Plasma Physics. This includes magnetohydrodynamic equilibrium in a toroidal device (Tokamak) modelled by the so-called Grad-Shafranov equation which can be formulated in terms of a semilinear elliptic equation of the form \eqref{eq1} (see e.g. \cite{Bi,MT,P}). In all these problems the nonlinear term $F$ plays a fundamental role in the corresponding physical law which explain the necessity of determining such expression. In this article we study the determination of such nonlinear terms from  measurements  given by measurement at one point at the boundary of the domain and general Dirichlet excitation lying in a space of dimension one.

The determination of nonlinear terms appearing in elliptic equations  has received a lot of attention this last decades. Most of the results in that direction are uniqueness results based on the linearization  method developed by \cite{Is1} and generalized by \cite{KLU}. In that direction we can mention the work of \cite{IN,IS} that considered the first results of determination of general semilinear terms appearing in an elliptic equation by applying the first order linearization. This approach was then extended by \cite{Is2,Is3} for the determination of semilinear terms depending only on the solution from partial data. More recently, many authors studied the determination of nonlinear terms from partial data or on a manifold by applying the higher order linearization technique and without being exhaustive we can mention the works of \cite{FO20,KKU,KU0,KU,LLLS,N,ST}. The linearization method has also been used for deriving stable determination of nonlinear terms depending only on the solution by several authors and one can refer for instance to the works \cite{CHY,K1,K2}.

Most of the above mentioned results require an important amount of data that are difficult to compute numerically. In addition, the linearization method used in these theoretical results is not yet well understood in the context of numerical reconstruction and, as far as we know, there has been no numerical reconstruction method for these theoretical  results. Some alternative approach  to the linearization methods have also been developed for proving the determination of a nonlinear term appearing in an elliptic equation. This includes the work of \cite{BV}, that strongly relies on the singularities of the domain $\Omega$ and does not work with smooth domains, and the approach of \cite{Ca,PR} solving this inverse problem with overspecified data. To the best of our knowledge, only the approach of \cite{Ca,PR} has been exploited for the derivation of a reconstruction algorithm  of a quasilinear term in \cite{EPS1,Kug}. We are not aware of any article in the mathematical literature studying the numerical reconstruction of a semilinear term appearing in an elliptic equation in accordance with the available theoretical results. The main goal of the present article is to prove theoretical results involving  data that can be exploited for numerical reconstruction and to use such data for the derivation of  a reconstruction algorithm based on Tikhonov regularization method. Specifically, an iterative thresholding algorithm has been developed and employed to address the nonlinear source inversion problem mentioned above. In recent years, this type of iterative approach has gained considerable attention in solving various inverse problems, including image processing \cite{LJY}, inverse source problems \cite{JFZ2015}, and coefficient identification problems for PDEs \cite{YLY}. For the convergence analysis of the iterative thresholding algorithm within a general framework, we can refer to \cite{DDD2004,DTV2007}. Our analysis is adapted to the determination of semilinear terms depending on the solutions. By employing the iterative thresholding algorithm, extensive numerical examples corroborate our theoretical analysis on the uniqueness and stability results. In particular, it has been demonstrated that a single point of measurement data is sufficient to effectively reconstruct the nonlinear source term.

\section{Main results}

In this section we state our main results which includes uniqueness and stability results. We start with a uniqueness results stated for a general class of semilinear terms.

\begin{theorem}\label{t1}   For $j=1,2$, let $F_j\in C^2(\R)$ with $F_j$ non-increasing and $F_j(0)=0$. Consider $u_{j,\delta}\in C^{2+\alpha}(\overline{\Omega})$ the solution of \eqref{eq1} for $F=F_j$, $\delta\in[0,1]$.
Then, for any arbitrary chosen $x_0\in\partial\Omega$ and $\delta_0\in(0,1]$, the condition
\bel{t1a}\partial_{\nu_a}u_{1,\delta}(x_0)=\partial_{\nu_a}u_{2,\delta}(x_0),\quad \delta\in[0,\delta_0]\ee
implies that one of the following conditions holds true:\\
(i) The map $F=F_1-F_2$ changes sign an infinite number of time on $[0,\delta_0 R]$.\\
(ii) $F_1=F_2$ on $[0,\delta_0 R]$.\end{theorem}
As a direct consequence of Theorem \ref{t1} we can prove the following.

\begin{cor}\label{c1} Let the condition of Theorem \ref{t1} be fulfilled and assume that, for $j=1,2$, $F_j$ is analytic on $\R$. Then condition \eqref{t1a} implies that $F_1=F_2$.
\end{cor}
This result can be improved in the following way.
\begin{cor}\label{c1c} Let the condition of Theorem \ref{t1} be fulfilled and assume that, for $j=1,2$, $F_j$ is analytic on $\R$. Fix $(\delta_k)_{k\in\mathbb N}$ a decreasing sequence   of $(0,1]$ that converges to zero. Then condition
\bel{c1ca}\partial_{\nu_a}u_{1,\delta_k}(x_0)=\partial_{\nu_a}u_{2,\delta_k}(x_0),\quad k\in\mathbb N\ee
 implies that $F_1=F_2$.
\end{cor}

We can also prove a uniqueness result from a single boundary measurement for some class of polynomial nonlinear terms.
 \begin{cor}\label{c3}  Let the condition of Theorem \ref{t1} be fulfilled with $g(\partial\Omega)=[0,1]$, i.e. $R=1$.
Assume that, for $j=1,2$,
\bel{c3a} F_j(s)=b_js^{\ell_j}h(s),\quad s\in\R,\ee
with $h\in C^2(\mathbb R)$ a function with no zero  on $[0,1]$, $\ell_j\in\mathbb N$, and
\bel{c3b} \min(|b_1|,|b_2|)=m>0,\quad \max(|b_1|,|b_2|)=M.\ee
Then, for any $\delta\in(0,m/M)$, the condition
\bel{c3c}\partial_{\nu_a}u_{1,\delta}(x_0)=\partial_{\nu_a}u_{2,\delta}(x_0)\ee
implies that $F_1=F_2$.
\end{cor}

The uniqueness result of Theorem \ref{t1} can be extended to the following stability result.

\begin{theorem}\label{t3} Let Theorem \ref{t1} be fulfilled and assume that there exists $M>0$ such that
\bel{t3a} \norm{F_j'}_{L^\infty(0,R)}\leq M,\quad j=1,2.\ee
Assume also that the map $F_1-F_2$ changes sign $N$ times. Then, for $x_0\in\partial\Omega$, there exists $C>0$ depending on $R$, $M$, $N$, $\Omega$, $x_0$ such that,
for $g(x_0)=R$, we have
\bel{t3b}\norm{F_1-F_2}_{L^\infty(0,R)}\leq C\left(\sup_{\delta\in[0,1]}|\partial_{\nu_a}u_{1,\delta}(x_0)-\partial_{\nu_a}u_{2,\delta}(x_0)|\right)^{\frac{1}{2^N}},\ee
while, if $g(x_0)<R$, we obtain
\bel{t3c}\norm{F_1-F_2}_{L^\infty(0,R)}\leq C\left(\sup_{\delta\in[0,1]}|\partial_{\nu_a}u_{1,\delta}(x_0)-\partial_{\nu_a}u_{2,\delta}(x_0)|\right)^{\frac{1}{(n+2)^N}}.\ee\end{theorem}

As a  consequence of this result, we can prove the following result stated with a single boundary measurement under a monotonicity assumption.

\begin{cor}\label{c2} Let the condition of Theorem \ref{t3} be fulfilled and assume that $F_1-F_2$ is of constant sign. Then, for $x_0\in\partial\Omega$, there exists $C>0$ depending on $R$, $M$, $N$, $\Omega$, $x_0$ such that,
for $g(x_0)=R$, we have
\bel{c2a}\norm{F_1-F_2}_{L^\infty(0,R)}\leq C|\partial_{\nu_a}u_{1,1}(x_0)-\partial_{\nu_a}u_{2,1}(x_0)|^{\frac{1}{2}},\ee
while, if $g(x_0)<R$, we obtain
\bel{c2b}\norm{F_1-F_2}_{L^\infty(0,R)}\leq C|\partial_{\nu_a}u_{1,1}(x_0)-\partial_{\nu_a}u_{2,1}(x_0)|^{\frac{1}{(n+2)}}.\ee

\end{cor}

Let us remark that in all the above mentioned results we consider the unique and stable determination of a semilinear term depending only on the solution from important restriction of the data used so far for this class of inverse problems. Namely, while all the results, that we are aware of, consider Dirichlet excitations lying on an infinite dimensional space, in this article we consider Dirichlet excitations lying in the space spanned by a single general Dirichlet data $g$ only subjected to the condition $g(\partial\Omega)=[0,R]$. In addition, in the spirit of the work \cite{N,ST} we restrict our measurements to an arbitrary point on the boundary of the domain. Note also that in Corollary \ref{c3} we show that our uniqueness results can be stated with a single Dirichlet excitation for the determination of  semilinear terms of the form \eqref{c3a}. This last result can be stated in terms of determination of a semilinear term from single boundary measurement and as far as we know it is the first result in that direction for elliptic equations. Our uniqueness results are also extended to H\"older stability results comparable to the one of \cite{CHY} for parabolic equations, where we replace the measurements on the full boundary of \cite{CHY}  by point measurements. We mention also that our results are all stated for general class of elliptic equations with variable coefficients.

In contrast of most results devoted to the determination of nonlinear terms appearing in elliptic equations, the proof of our results are not based on the linearization method of \cite{Is1,KLU}. Instead, we develop a new  approach combining the monotonicity arguments  used for solving inverse source problems for elliptic equations and different properties of solutions of \eqref{eq1}. The link of our method with inverse source problems will be exploited in Section 6 for the numerical reconstruction of a semilinear term. The implementation of a diverse array of numerical examples has convincingly demonstrated that our algorithm is capable of effectively inverting the semilinear term using only a single point of measurement data. As far as we know, we obtain in this article the first reconstruction algorithm for the determination of a semilinear term appearing in an elliptic equation and our results are supplemented by the theoretical uniqueness and stability stated in Theorem \ref{t1},  \ref{t3} and their consequences.

This article is organized as follows. In Section 3 we recall some preliminary properties of solutions of elliptic equations, including properties of Poisson kernel, of some class of linear elliptic equations. Section 4 will be devoted to the proof of our uniqueness results while in Section 5 we show our stability result.  Finally, Section 6 will be devoted to the numerical analysis of this problem including the algorithm that we develop in Section 6.1 based on our theoretical results.
\section{Preliminary properties}

In this section we recall some preliminaries properties required for the proof of the main results of this article. Namely, fixing $q\in C^1(\overline{\Omega})$ a non-negative function and following \cite[Section 5]{MT},  we can consider $P(x,y)$, $x\in\Omega$, $y\in\partial\Omega$, the Poisson kernel of the equation
 $$-|a|^{\frac{1}{2}}\sum_{i,j=1}^n \partial_{x_i}
\left( |a|^{-\frac{1}{2}}a_{i,j}(x) \partial_{x_j} v(x) \right)+q(x)v(x)=0$$
with homogeneous Dirichlet boundary condition. We recall that the Poisson kernel $P$ is a  map defined on $\overline{\Omega}\times\partial\Omega\setminus\{(x,y)\in\overline{\Omega}\times\partial\Omega:\ x=y\}$ such that, for any $\phi\in C^{2+\alpha}(\partial\Omega)$, the solution of the boundary value problem
$$\left\{
\begin{array}{ll}
-|a|^{\frac{1}{2}}\sum_{i,j=1}^n \partial_{x_i}
\left( |a|^{-\frac{1}{2}}a_{i,j}(x) \partial_{x_j} v \right)+qv=0 & \mbox{in}\ \Omega ,
\\
v=\phi &\mbox{on}\ \partial\Omega,
\end{array}
\right.$$
is given by
$$v(x)=\int_{\partial\Omega}P(x,y)\phi(y)|a|^{-\frac{1}{2}}(y)d\sigma(y).$$
We fix $b=(b_{i,j})_{1\leq i,j\leq n}\in  C^2(\overline{\Omega};\R^{n^2})$ defined by $b=a^{-1}$. Following \cite[Section 5]{MiT} (see also \cite[Section 5]{HMT}), there exists a purely dimensional constant $C_n>0$ such that, for all  $x\in\Omega$, $y\in\partial\Omega$, we have
$$\abs{P(x,y)-C_n|a|^{\frac{1}{2}}(x)\frac{\left\langle \nu_a(y),(y-x)\right\rangle_{b(x)}}{|x-y|_{b(x)}^n}}\leq C|x-y|_{b(x)}^{1-n},$$
where $C>0$ is a constant depending on $\Omega$, $a$ and $\norm{q}_{L^\infty(\Omega)}$, and
$$\left\langle X,Y\right\rangle_{b(x)}=\sum_{i,j=1}^nb_{ij}(x)X_iY_j,\ |X|_{b(x)}=\sqrt{\left\langle X,X\right\rangle_{b(x)}},\quad X=(X_1,\ldots,X_n),\ Y=(Y_1,\ldots,Y_n)\in\R^n,$$
$$\nu_a(y)=\left(\sum_{i=1}^na_{i1}(y)\nu_i(y),\ldots,\sum_{i=1}^na_{in}(y)\nu_i(y)\right)^T.$$
From the above properties and the maximum principle, in a similar way to \cite[Theorem 1]{Kr}, we can find $c_1,c_2>0$ and $\epsilon_0>0$, depending on $\Omega$, $a$ and $\norm{q}_{L^\infty(\Omega)}$, such that
\bel{poisson2}0\leq P(x,y)\leq c_2\frac{\textrm{dist}(x,\partial\Omega)}{|x-y|^{n}},\quad x\in\Omega,\ y\in\partial\Omega,\ee
\bel{poisson3} c_1\frac{\textrm{dist}(x,\partial\Omega)}{|x-y|^{n}}\leq P(x,y),\quad x\in\Omega,\ y\in\partial\Omega,\ |x-y|\leq\epsilon_0.\ee

Following \cite[Lemma 2.1]{ST}, we can prove the following result.

\begin{lem}\label{l3} Fix $x_0\in\partial\Omega$ and $w\in C^2(\overline{\Omega})\cap H^1_0(\Omega)$. Then the following identity
\bel{l3a}\int_\Omega \left(-|a|^{\frac{1}{2}}\sum_{i,j=1}^n \partial_{x_i}
\left( |a|^{-\frac{1}{2}}a_{i,j}(x) \partial_{x_j} w \right)+qw\right)P(x,x_0)|a|^{-\frac{1}{2}}dx=-\partial_{\nu_a}w(x_0)\ee
holds true.\end{lem}
\begin{proof} The proof of this result is inspired by \cite[Lemma 2.1]{ST} where a similar result was proved for $q=0$.
For any finite Borel measure $\mu$ on $\partial\Omega$ we denote by $|\mu|$ the absolute value of $\mu$ on $\partial\Omega$.
By applying a partition of unity, boundary flattening transformations
and convolution approximation, we can define a sequence $(\phi_k)_{k\in\mathbb N}$ of non-negative functions lying in $C^{2+\alpha}(\partial\Omega)$ such that
\bel{l3b}\lim_{k\to\infty}\abs{\phi_k|a|^{-\frac{1}{2}}d\sigma-d\delta_{x_0}}(\partial\Omega)=0,\ee
with $d\delta_{x_0}$ the delta measure at $x=x_0$ defined by
$$\int_{\partial\Omega}\psi(y) d\delta_{x_0}(y)=\psi(x_0),\quad \psi\in C(\partial\Omega).$$
For $k\in\mathbb N$, let us consider $v_k\in C^{2+\alpha}(\overline{\Omega})$ solving the boundary value problem
$$\left\{
\begin{array}{ll}
-|a|^{\frac{1}{2}}\sum_{i,j=1}^n \partial_{x_i}
\left( |a|^{-\frac{1}{2}}a_{i,j}(x) \partial_{x_j} v_k \right)+qv_k=0 & \mbox{in}\ \Omega ,
\\
v_k=\phi_k &\mbox{on}\ \partial\Omega.
\end{array}
\right.$$
Integrating by parts, we find
\bel{l3e}\int_\Omega \left(-|a|^{\frac{1}{2}}\sum_{i,j=1}^n \partial_{x_i}
\left( |a|^{-\frac{1}{2}}a_{i,j}(x) \partial_{x_j} w \right)+qw\right)v_k|a|^{-\frac{1}{2}}dx=-\int_{\partial\Omega}\partial_{\nu_a}w \phi_k|a|^{-\frac{1}{2}}(x)d\sigma(x),\quad k\in\mathbb N.\ee
Recall that
$$v_k(x)=\int_{\partial\Omega}P(x,y)\phi_k(y)|a|^{-\frac{1}{2}}(y)d\sigma(y),\quad x\in\Omega,\ k\in\mathbb N.$$
We fix $\tau>0$ and $\Omega_\tau:=\{x\in\Omega:\ \textrm{dist}(x,\partial\Omega)>\tau\}$. Applying \eqref{poisson2}, for all $k\in\mathbb N$, we get
$$\begin{aligned} \norm{v_k-P(\cdot,x_0)}_{L^1(\Omega_\tau)}&\leq \int_{\Omega_\tau}\abs{\int_{\partial\Omega}P(x,y)(\phi_{k}(y)d\sigma(y)-d\delta_{x_0}(y))}\\
&\leq \int_{\Omega_\tau}\left(\int_{\partial\Omega}|P(x,y)|\abs{\phi_{k}d\sigma-d\delta_{x_0}}(y)\right)dx\\
&\leq C\int_{\partial\Omega}\left(\int_{\Omega_\tau}|x-y|^{-(n-1)}dx\right)\abs{\phi_{k}d\sigma-d\delta_{x_0}}(y)\\
& \leq C\left(\sup_{y\in\partial\Omega}\int_{\Omega_\tau}|x-y|^{-(n-1)}dx\right)\abs{\phi_kd\sigma-d\delta_{x_0}}(\partial\Omega)\\
& \leq C\left(\sup_{y\in\partial\Omega}\int_{\Omega}|x-y|^{-(n-1)}dx\right)\abs{\phi_kd\sigma-d\delta_{x_0}}(\partial\Omega).\end{aligned}$$
Moreover, one can check that
$$\sup_{y\in\partial\Omega}\int_{\Omega}|x-y|^{-(n-1)}dx<\infty.$$
Thus, sending $\tau\to0$, we get
$$\norm{v_k-P(\cdot,x_0)}_{L^1(\Omega)}\leq C\abs{\phi_kd\sigma-d\delta_{x_0}}(\partial\Omega),\quad k\in\mathbb N$$
 and applying \eqref{l3b}, we find
$$\lim_{k\to\infty}\norm{v_k-P(\cdot,x_0)}_{L^1(\Omega)}=0.$$
It follows that
$$\begin{aligned}&\lim_{k\to\infty}\int_\Omega \left(-|a|^{\frac{1}{2}}\sum_{i,j=1}^n \partial_{x_i}
\left( |a|^{-\frac{1}{2}}a_{i,j}(x) \partial_{x_j} w \right)+qw\right)v_k|a|^{-\frac{1}{2}}dx\\
&=\int_\Omega \left(-|a|^{\frac{1}{2}}\sum_{i,j=1}^n \partial_{x_i}
\left( |a|^{-\frac{1}{2}}a_{i,j}(x) \partial_{x_j} w \right)+qw\right)P(x,x_0)|a|^{-\frac{1}{2}}dx\end{aligned}$$
and \eqref{l3b} implies that
$$\lim_{k\to\infty}\int_{\partial\Omega}\partial_{\nu_a}w \phi_k|a|^{-\frac{1}{2}}d\sigma(x)=\int_{\partial\Omega}\partial_{\nu_a}w(x) d\delta_{x_0}(x)=\partial_{\nu_a}w(x_0).$$
Combining this with \eqref{l3e}, we get \eqref{l3a}. This completes the proof of the lemma.\end{proof}

\section{Proof of Theorem \ref{t1} and Corollary \ref{c1c}, \ref{c3}}
\subsection{Proof of Theorem \ref{t1}}
We will prove this result by contradiction. For this purpose, we assume that \eqref{t1a} is fulfilled but  the sign of $F=F_1-F_2$ changes a finite number of times on $[0,\delta_0R]$ and $F|_{[0,\delta_0R]}\not\equiv0$. Therefore,  there exists $\delta_1\in(0,\delta_0]$ such that  $F$ is of constant sign on $[0,\delta_1R]$ and $F|_{[0,\delta_1R]}\not\equiv0$. Without loss of generality we may assume that $F\geq0$ on $[0,\delta_1R]$. Fixing $u_\delta=u_{1,\delta}-u_{2,\delta}$ we deduce that $u_\delta$ solves
$$\left\{
\begin{array}{ll}
-|a|^{\frac{1}{2}}\sum_{i,j=1}^n \partial_{x_i}
\left( |a|^{-\frac{1}{2}}a_{i,j}(x) \partial_{x_j} u_\delta(x) \right)+q_\delta(x)u_\delta(x)=F_1(u_{2,\delta}(x))-F_2(u_{2,\delta}(x)) & \mbox{in}\ \Omega ,
\\
u_\delta=0 &\mbox{on}\ \partial\Omega,
\end{array}
\right.$$
with
$$q_\delta(x)=-\int_0^1F_1'(su_{1,\delta}(x)+(1-s)u_{2,\delta}(x))ds,\quad x\in\Omega.$$
Recalling that $F_1'\leq0$ and $F_1\in C^2(\R)$, we deduce that $q_\delta\in C^1(\overline{\Omega})$ and $q_\delta\geq0$. Therefore,   we can consider $P_\delta(x,y)$, $x\in\Omega$, $y\in\partial\Omega$, the Poisson kernel of the equation
 $$-|a|^{\frac{1}{2}}\sum_{i,j=1}^n \partial_{x_i}
\left(|a|^{-\frac{1}{2}} a_{i,j}(x) \partial_{x_j} v(x) \right)+q_\delta(x)v(x)=0$$
with homogeneous Dirichlet boundary condition. Applying Lemma \ref{l3}, we obtain
$$\begin{aligned} &\int_\Omega(F_1(u_{2,\delta}(x))-F_2(u_{2,\delta}(x)))P_\delta(x,x_0)|a|^{-\frac{1}{2}}dx\\
&=\int_\Omega\left(-|a|^{\frac{1}{2}}\sum_{i,j=1}^n \partial_{x_i}
\left(|a|^{-\frac{1}{2}} a_{i,j}(x) \partial_{x_j} u_\delta(x) \right)+q(x)u_\delta(x)\right)P_\delta(x,x_0)|a|^{-\frac{1}{2}}dx\\
&=-\partial_{\nu_a}u_{\delta}(x_0)=\partial_{\nu_a}u_{2,\delta}(x_0)-\partial_{\nu_a}u_{1,\delta}(x_0).\end{aligned}$$
Then, \eqref{t1a} implies
\bel{t1b}\int_\Omega(F_1(u_{2,\delta}(x))-F_2(u_{2,\delta}(x)))P_\delta(x,x_0)|a|^{-\frac{1}{2}}dx=0,\quad \delta\in[0,\delta_1].\ee
On the other hand, for $j=1,2$ and $\delta\in(0,1]$, using the fact that $F_2(0)=0$, we deduce that $u_{2,\delta}$ solves the problem
$$\left\{
\begin{array}{ll}
-|a|^{\frac{1}{2}}\sum_{i,j=1}^n \partial_{x_i}
\left(|a|^{-\frac{1}{2}} a_{i,j}(x) \partial_{x_j} u_{2,\delta}(x) \right)+q_2(x)u_{2,\delta}(x)=0 & \mbox{in}\ \Omega ,
\\
u_\delta=\delta g &\mbox{on}\ \partial\Omega,
\end{array}
\right.$$
with
$$q_2(x)=-\int_0^1F_2'(su_{2,\delta}(x))ds,\quad x\in\Omega.$$
Using the fact that $F_2'\leq0$ and applying the maximum principle we deduce that
$$0\leq u_{2,\delta}(x)\leq \delta R,\quad x\in\Omega.$$
It follows that
$$F_1(u_{2,\delta_1}(x))-F_2(u_{2,\delta_1}(x))\geq 0,\quad x\in\Omega.$$
In the same way, from \eqref{poisson2} we know that $P_{\delta_1}(x,x_0)\geq0$, $x\in\Omega$,
and \eqref{t1b} implies that
\bel{t1c}(F_1(u_{2,\delta_1}(x))-F_2(u_{2,\delta_1}(x)))P_{\delta_1}(x,x_0)=0,\quad x\in\Omega.\ee
By unique continuation properties for elliptic equations, we know that the map $x\mapsto P(x,x_0)$ can not vanishes on any open subset of $\Omega$ and by continuity of the map
$x\mapsto F_1(u_{2,\delta_1}(x))-F_2(u_{2,\delta_1}(x))$ on $\overline{\Omega}$,  we obtain
$$F_1(u_{2,\delta_1}(x))-F_2(u_{2,\delta_1}(x))=0,\quad x\in\overline{\Omega}.$$
Therefore, we have
$$F_1(\delta_1g(x))=F_2(\delta_1g(x)),\quad x\in\partial\Omega$$
and, recalling that $g(\partial\Omega)=[0,R]$, it follows that $F_1=F_2$ on $[0,\delta_1R]$ which contradicts the fact that $F|_{[0,\delta_1R]}\not\equiv0$. This complete the proof of the theorem.
\subsection{Proof of  Corollary \ref{c1c}}
We assume that \eqref{c1c} is fulfilled. Repeating the arguments of Theorem \ref{t1} and applying \eqref{c3c}, we find
$$\int_\Omega F(u_{2,\delta_k}(x))P_k(x,x_0)|a|^{-\frac{1}{2}}dx=\int_\Omega(F_1(u_{2,\delta_k}(x))-F_2(u_{2,\delta_k}(x)))P_k(x,x_0)|a|^{-\frac{1}{2}}dx=0,\quad k\in\mathbb N,$$
where, for all $k\in\mathbb N$,  $P_k$ denotes the Poisson kernel of the equation
 $$-|a|^{\frac{1}{2}}\sum_{i,j=1}^n \partial_{x_i}
\left( |a|^{-\frac{1}{2}}a_{i,j}(x) \partial_{x_j} v(x) \right)+q_k(x)v(x)=0,$$
$$q_k(x)=-\int_0^1F_1'(su_{1,\delta_k}(x)+(1-s)u_{2,\delta_k}(x))ds,\quad x\in\Omega,$$
with homogeneous Dirichlet boundary condition. Fixing $F=F_1-F_2$, we obtain
\bel{c1cb}\int_\Omega F(u_{2,\delta_k}(x))P_k(x,x_0)|a|^{-\frac{1}{2}}dx=0,\quad k\in\mathbb N\ee
and, in a similar way to Theorem \ref{t1}, applying the maximum principle we get
\bel{c1cc}0\leq u_{2,\delta_k}(x)\leq \delta_k R,\quad x\in\Omega,\ k\in\mathbb N.\ee
Using the fact that $F_1$ and $F_2$ are analytic on $\R$, we deduce that $F$ is analytic on $\R$. Moreover, recalling that the sequence $(\delta_k)_{k\in\mathbb N}$ is a decreasing sequence that  converges to zero and applying the isolated zero theorem, we deduce that there exists $k_0\in\mathbb N$ such that $F$ is of constant sign on $[0,\delta_{k_0}R]$. Combining this with \eqref{c1cc}, we obtain that the map $x\mapsto F(u_{2,\delta_{k_0}}(x))$ is of constant sign on $\Omega$ and \eqref{c1cb} implies that
$$F(u_{2,\delta_{k_0}}(x))P_{k_0}(x,x_0)=0,\quad x\in\Omega.$$
Then, repeating the arguments used in the proof of Theorem \ref{t1}, we obtain
$$F(u_{2,\delta_{k_0}}(x))=0,\quad x\in\overline{\Omega}$$
which implies that
$$F(\delta_{k_0}g(x))=F(u_{2,\delta_{k_0}}(x))=0,\quad x\in\partial\Omega.$$
Combining this with the fact that $g(\partial\Omega)=[0,R]$ and applying again the isolated zero theorem, we deduce that $F\equiv0$ which implies that $F_1=F_2$. This completes the proof of the corollary.

\subsection{Proof of  Corollary \ref{c3}}
We assume that \eqref{c3c} is fulfilled for some $\delta\in(0,m/M)$ and we will prove that this condition implies that $F_1=F_2$. We will start by proving that $\ell_1=\ell_2$. For this purpose let us assume the contrary. Without loss of generality we may assume that $\ell_1<\ell_2$. Fixing $F=F_1-F_2$ and recalling that $|h|>0$ on $[0,\delta]\subset[0,1]$, for all $s\in(0,\delta)$, we have
$$|F_2(s)|\leq |b_2||s|^{\ell_2}|h(s)|\leq M |s|^{\ell_2}|h(s)|\leq\frac{M}{m} \delta^{\ell_2-\ell_1}|b_1||s|^{\ell_1}|h(s)|.$$
Now recalling that
$$\frac{M}{m} \delta^{\ell_2-\ell_1}< \frac{M}{m} \left(\frac{m}{M}\right)^{\ell_2-\ell_1}\leq \left(\frac{m}{M}\right)^{\ell_2-1-\ell_1}\leq 1,$$
we deduce that
$$|F_2(s)|\leq\frac{M}{m} \delta^{\ell_2-\ell_1}|b_1||s|^{\ell_1}|h(s)|<|b_1||s|^{\ell_1}|h(s)|=|F_1(s)|,\quad s\in(0,\delta].$$
Thus, we have
$$|F(s)|\geq|F_1(s)|-|F_2(s)|>0,\quad s\in(0,\delta],$$
which implies that
\bel{c3d}|F(s)|>0,\quad s\in(0,\delta).\ee
This implies that $F$ is of constant sign on $[0,\delta]$ and without loss of generality we may assume that $F\geq0$ on $[0,\delta]$. Repeating the arguments of Theorem \ref{t1} and applying \eqref{c3c}, we obtain that
$$\int_\Omega F(u_{2,\delta}(x))P_\delta(x,x_0)|a|^{-\frac{1}{2}}dx=\int_\Omega(F_1(u_{2,\delta}(x))-F_2(u_{2,\delta}(x)))P_\delta(x,x_0)|a|^{-\frac{1}{2}}dx=0,$$
where $P_\delta$ denotes the Poisson kernel of the equation
 $$-|a|^{\frac{1}{2}}\sum_{i,j=1}^n \partial_{x_i}
\left( |a|^{-\frac{1}{2}}a_{i,j}(x) \partial_{x_j} v(x) \right)+q_\delta(x)v(x)=0,$$
$$q_\delta(x)=-\int_0^1F_1'(su_{1,\delta}(x)+(1-s)u_{2,\delta}(x))ds,\quad x\in\Omega,$$
with homogeneous Dirichlet boundary condition. Combining this with the fact that
$$0\leq u_{2,\delta}(x)\leq \delta,\quad F(s)\geq0,\quad x\in\Omega,\ s\in[0,\delta],$$
and repeating the arguments used in Theorem \ref{t1}, we deduce that
$$F(s)=0,\quad s\in[0,\delta].$$
This contradicts \eqref{c3d} and we deduce that $\ell_1=\ell_2$.

In order to complete the proof of the corollary, we only need to show that $b_1=b_2$. Recalling that
$$F(s)=(b_1-b_2)s^{\ell_1}h(s),\quad s\in\mathbb R,$$
and using the fact that $h$ is of constant sign on $[0,1]$, we deduce that $F$ is of constant sign on $[0,\delta]$ and repeating the above argumentation, we deduce that
$$F(s)=0,\quad s\in[0,\delta].$$
and,  combining this with the fact that $|h|>0$ on $[0,\delta]$, we conclude that  $b_1=b_2$. It follows that  $F_1=F_2$, which completes the proof of the corollary.

\section{Proof of Theorem \ref{t3} and Corollary \ref{c2}}
\subsection{Proof of Theorem \ref{t3}}
We fix $F=F_1-F_2$ and we consider $0=s_0<s_1<\ldots<s_N\leq R$ the $N$ points of $[0,R]$ where $F$ changes sign. We fix also $0<r_1<\ldots<r_N\leq R$ such that
$$r_j\in(s_j,s_{j+1}),\quad \sup_{s\in[s_j,s_{j+1}]}|F(s)|=|F(r_j)|,\quad j=0,\ldots,N-1.$$

Without loss of generality, we may assume that $F\geq0$ on $[0,s_1]$. In a similar way to Theorem \ref{t1}, we denote by $P_\delta(x,y)$, $x\in\Omega$, $y\in\partial\Omega$, the Poisson kernel of the equation
 \bel{q1}-|a|^{\frac{1}{2}}\sum_{i,j=1}^n \partial_{x_i}
\left(|a|^{-\frac{1}{2}} a_{i,j}(x) \partial_{x_j} v(x) \right)+q_\delta(x)v(x)=0\ee
with
\bel{q2}q_\delta(x)=-\int_0^1F_1'(su_{1,\delta}(x)+(1-s)u_{2,\delta}(x))ds,\quad x\in\Omega.\ee
Repeating the argumentation of Theorem \ref{t1}, we obtain the following identity
\bel{t3e}\int_\Omega(F_1(u_{2,\delta}(x))-F_2(u_{2,\delta}(x)))P_\delta(x,x_0)|a|^{-\frac{1}{2}}dx=\partial_{\nu_a}u_{2,\delta}(x_0)-\partial_{\nu_a}u_{1,\delta}(x_0),\quad \delta\in[0,r_1/R]\ee
where we recall that
\bel{q3}0\leq u_{j,\delta}(x)\leq \delta R\leq r_1,\quad x\in\Omega,\ \delta\in[0,r_1/R].\ee
For  $\tau\in[0,M]$, let us  consider $v_{\delta,\tau}$ the solution of the boundary value problem
$$\left\{
\begin{array}{ll}
-|a|^{\frac{1}{2}}\sum_{i,j=1}^n \partial_{x_i}
\left( |a|^{-\frac{1}{2}}a_{i,j}(x) \partial_{x_j} v_\tau(x) \right)+\tau v_\tau(x)=0 & \mbox{in}\ \Omega ,
\\
v_\tau=\delta g &\mbox{on}\ \partial\Omega,
\end{array}
\right.$$
Fixing $w_\tau=u_{2,\delta}-v_{\delta,\tau}$, we deduce that $w_\tau$ solves
$$\left\{
\begin{array}{ll}
-\sum_{i,j=1}^n \partial_{x_i}
\left( a_{i,j}(x) \partial_{x_j} w_\tau(x) \right)+\tau w_\tau(x)=(\tau+q_2(x))u_{2,\delta}(x) & \mbox{in}\ \Omega ,
\\
v=\delta g &\mbox{on}\ \partial\Omega,
\end{array}
\right.$$
with
$$q_2(x)=\int_0^1F_2'(su_{2,\delta}(x))ds,\quad x\in\Omega.$$
Using \eqref{t3a}, we deduce that
$$\tau-M\leq \tau+q_2(x)\leq \tau,\quad x\in\Omega$$
and, applying the maximum principle, we find $w_0\geq0$ and $w_M\leq0$. Therefore, we have
$$v_{\delta,M}(x)\leq u_{2,\delta}(x)\leq v_{\delta,0}(x),\quad x\in\Omega.$$
In addition, one can check that the map $\tau\mapsto v_{\delta,\tau}$ is continuous from $[0,M]$ to $C^{2+\alpha}(\overline{\Omega})$ and the mean value theorem implies that, for all $x\in\Omega$, there exists $\tau_x\in[0,M]$ such that $u_{2,\delta}(x)=v_{\delta,\tau_x}(x)$. In particular, we have
$$\inf_{\tau\in[0,M]}F(v_{\delta,\tau}(x))\leq F(v_{\delta,\tau_x}(x))=F(u_{2,\delta}(x)),\quad x\in\Omega$$
and \eqref{t3e} combined with the fact that $P_\delta(x,x_0)\geq0$, $x\in\Omega$, imply
\bel{t3f}\begin{aligned}\int_\Omega \inf_{\tau\in[0,M]} F(v_{\delta,\tau}(x))P_\delta(x,x_0)|a|^{-\frac{1}{2}}dx&\leq \int_\Omega F(u_{2,\delta}(x))P_\delta(x,x_0)|a|^{-\frac{1}{2}}dx\\
&\leq|\partial_{\nu_a}u_{1,\delta}(x_0)-\partial_{\nu_a}u_{2,\delta}(x_0)|,\quad \delta\in[0,r_1/R].\end{aligned}\ee

We fix $x_1\in\partial\Omega$ such that $g(x_1)=R$. We consider two situations, $x_0=x_1$ and $x_0\neq x_1$.\\

\textbf{Case 1}: $x_1=x_0$.

Fix $\epsilon>0$ and consider the set $\mathcal O_\epsilon:= \{x\in\Omega:\ |x-x_1|<\epsilon\}$. Applying the maximum principle on can check that
$$0\leq v_{\tau,\delta}(x)\leq \delta R\leq r_1,\quad x\in\Omega,\ \delta\in[0,r_1/R],\ \tau\in[0,M]$$
which implies that
$$F(v_{\tau,\delta}(x))\geq0,\quad x\in\Omega,\ \delta\in[0,r_1/R],\ \tau\in[0,M].$$
Combining this with \eqref{t3f}, we get
\bel{t3g}\begin{aligned}\int_{\mathcal O_\epsilon} \inf_{\tau\in[0,M]} F(v_{\delta,\tau}(x))P_\delta(x,x_0)|a|^{-\frac{1}{2}}dx&\leq\int_\Omega \inf_{\tau\in[0,M]} F(v_{\delta,\tau}(x))P_\delta(x,x_0)|a|^{-\frac{1}{2}}dx\\
&\leq|\partial_{\nu_a}u_{1,\delta}(x_0)-\partial_{\nu_a}u_{2,\delta}(x_0)|,\quad \delta\in[0,r_1/R].\end{aligned}\ee
Moreover, applying \eqref{t3a} and using the fact that the map $\R\ni\lambda\mapsto v_{\lambda,\tau}$ is linear, for all $\tau\in[0,M]$, we find
\bel{t2aaa}\begin{aligned}\abs{F(v_{\delta,\tau}(x))-F(\delta R)}&\leq M\abs{v_{\delta,\tau}(x)-v_{\delta,\tau}(x_1)}\\
&\leq M\norm{v_{\delta,\tau}}_{w^{1,\infty}(\Omega)}|x-x_1|\\
&\leq M\delta \sup_{\tau\in[0,M]}\norm{v_{1,\tau}}_{w^{1,\infty}(\Omega)}|x-x_1|\\
&\leq C\epsilon,\quad x\in\mathcal O_\epsilon\end{aligned}\ee
with $C>0$ a constant depending on $M$, $g$, $\Omega$, $R$. Thus, recalling that the map $\tau\mapsto v_{\delta,\tau}$ is continuous from $[0,M]$ to $C^{2+\alpha}(\overline{\Omega})$ and $F\in C^1(\R)$, for all $x\in\overline{\Omega}$ we can find $\tau_x\in[0,M]$ such that
$$\inf_{\tau\in[0,M]} F(v_{\delta,\tau}(x))=F(v_{\delta,\tau_x}(x)).$$
Combining this with \eqref{t2aaa}, we obtain
$$\abs{\inf_{\tau\in[0,M]} F(v_{\delta,\tau}(x))-F(\delta R)}\leq C\epsilon,\quad x\in\mathcal O_\epsilon$$
and it follows
$$\begin{aligned} &\int_{\mathcal O_\epsilon}\inf_{\tau\in[0,M]} F(v_{\delta,\tau}(x))P_\delta(x,x_0)|a|^{-\frac{1}{2}}dx\\
&\geq F(\delta R)\int_{\mathcal O_\epsilon} P_\delta(x,x_0)|a|^{-\frac{1}{2}}dx-\int_{\mathcal O_\epsilon} \abs{\inf_{\tau\in[0,M]} F(v_{\delta,\tau}(x))-F(\delta R)}P_\delta(x,x_0)|a|^{-\frac{1}{2}}dx\\
&\leq \left(F(\delta R)-C\epsilon\right)\int_{\mathcal O_\epsilon} P_\delta(x,x_0)|a|^{-\frac{1}{2}}dx.\end{aligned}$$
On the other hand, in view of \eqref{t3a} and \eqref{q3}, the map $q_\delta$, $\delta\in[0,1]$, defined by \eqref{q2} satisfies the estimate
$$|q_\delta(x)|\leq\int_0^1|F_1'(su_{1,\delta}(x)+(1-s)u_{2,\delta}(x))|ds\leq \sup_{r\in[0,R]}|F_1'(r)|\leq M,\quad x\in\Omega,\ \delta\in[0,1].$$
Combing this estimate with the fact that $P_\delta(x,y)$, $x\in\Omega$, $y\in\partial\Omega$, is the Poisson kernel of the equation \eqref{q1} and applying \eqref{poisson3}, we obtain
$$P_\delta(x,x_0)=P_\delta(x,x_1)\geq C\frac{\textrm{dist}(x,\partial\Omega)}{|x-x_1|^n}\geq \frac{C}{|x-x_1|^{n-1}},\quad x\in\mathcal O_\epsilon,$$
with $C>0$ a constant depending on $M$, $a$ and  $\Omega$.
Then, applying \eqref{t3g}, we find
$$\left(F(\delta R)-C\epsilon\right)\int_{\mathcal O_\epsilon}\frac{dx}{|x-x_1|^{n-1}}\leq C|\partial_{\nu_a}u_{1,\delta}(x_0)-\partial_{\nu_a}u_{2,\delta}(x_0)|,\quad \delta\in[0,r_1/R]$$
and it follows that
$$|F(\delta R)|\leq C(\epsilon^{-1}|\partial_{\nu_a}u_{1,\delta}(x_0)-\partial_{\nu_a}u_{2,\delta}(x_0)|+\epsilon),\quad \delta\in[0,r_1/R],$$
$C>0$ a constant depending on $M$, $a$, $g$, $\Omega$, $R$.
Therefore, choosing $\epsilon=|\partial_{\nu_a}u_{1,\delta}(x_0)-\partial_{\nu_a}u_{2,\delta}(x_0)|^{\frac{1}{2}}$ and $s=\delta R$, we get
$$|F(\delta R)|\leq C|\partial_{\nu_a}u_{1,\delta}(x_0)-\partial_{\nu_a}u_{2,\delta}(x_0)|^{\frac{1}{2}},\quad s\in[0,r_1]$$
and it follows that
\bel{t3h}|F(s)|\leq |F(r_1)|\leq C\left(\sup_{\delta\in[0,1]}|\partial_{\nu_a}u_{1,\delta}(x_0)-\partial_{\nu_a}u_{2,\delta}(x_0)|\right)^{\frac{1}{2}},\quad s\in[0,s_1].\ee
Repeating the above argumentation, we obtain
$$\int_\Omega F(u_{2,\delta}(x))P_\delta(x,x_0)|a|^{-\frac{1}{2}}dx=\partial_{\nu_a}u_{1,\delta}(x_0)-\partial_{\nu_a}u_{2,\delta}(x_0),\quad \delta\in[s_1/R,r_2/R]$$
Thus, we have
$$\abs{\int_{u_{2,\delta}\geq s_1} F(u_{2,\delta}(x))P_\delta(x,x_0)|a|^{-\frac{1}{2}}dx}\leq \abs{\partial_{\nu_a}u_{1,\delta}(x_0)-\partial_{\nu_a}u_{2,\delta}(x_0)}+\int_{u_{2,\delta}\leq s_1} |F(u_{2,\delta}(x))|P_\delta(x,x_0)|a|^{-\frac{1}{2}}dx.$$
Using the fact that $F\leq 0$ on $[s_1,s_2]$, for all $\delta\in[s_1/R,r_2/R]$, we get
\bel{t3i}\int_{u_{2,\delta}\geq s_1} |F(u_{2,\delta}(x))|P_\delta(x,x_0)|a|^{-\frac{1}{2}}dx\leq \abs{\partial_{\nu_a}u_{1,\delta}(x_0)-\partial_{\nu_a}u_{2,\delta}(x_0)}+\int_{u_{2,\delta}\leq s_1} |F(u_{2,\delta}(x))|P_\delta(x,x_0)|a|^{-\frac{1}{2}}dx.\ee
Applying \eqref{t3h}, we find
$$|F(u_{2,\delta}(x))|\leq C\left(\sup_{\delta\in[0,1]}|\partial_{\nu_a}u_{1,\delta}(x_0)-\partial_{\nu_a}u_{2,\delta}(x_0)|\right)^{\frac{1}{2}},\quad x\in \{y\in\Omega:\  u_{2,\delta}(y)\leq s_1\}.$$
Combining this with \eqref{poisson2}, we find
$$\begin{aligned}\int_{u_{2,\delta}\leq s_1} |F(u_{2,\delta}(x))|P_\delta(x,x_0)|a|^{-\frac{1}{2}}dx&\leq C\left(\int_{u_{2,\delta}\leq s_1} P_\delta(x,x_0)dx\right)\left(\sup_{\delta\in[0,1]}|\partial_{\nu_a}u_{1,\delta}(x_0)-\partial_{\nu_a}u_{2,\delta}(x_0)|\right)^{\frac{1}{2}}\\
&\leq C\left(\int_{\Omega} |x-x_0|^{1-n}dx\right)\left(\sup_{\delta\in[0,1]}|\partial_{\nu_a}u_{1,\delta}(x_0)-\partial_{\nu_a}u_{2,\delta}(x_0)|\right)^{\frac{1}{2}}\\
&\leq C\left(\sup_{\delta\in[0,1]}|\partial_{\nu_a}u_{1,\delta}(x_0)-\partial_{\nu_a}u_{2,\delta}(x_0)|\right)^{\frac{1}{2}},\end{aligned}$$
with $C>0$ depending on $\Omega$, $M$ and $a$. Combining this with \eqref{t3i}, we get
$$\int_{u_{2,\delta}\geq s_1} |F(u_{2,\delta}(x))|P_\delta(x,x_0)|a|^{-\frac{1}{2}}dx\leq C\left(\sup_{\delta\in[0,1]}|\partial_{\nu_a}u_{1,\delta}(x_0)-\partial_{\nu_a}u_{2,\delta}(x_0)|\right)^{\frac{1}{2}},\quad \delta\in[s_1/R,r_2/R].$$
Now,  using the fact that  $u_{2,\delta}(x_1)=\delta R>s_1$, $\delta\in(s_1/R,r_2/R]$, and choosing $\epsilon>0$ sufficiently small, we get $\mathcal O_\epsilon\subset \{y\in\Omega:\  u_{2,\delta}(y)\geq s_1\}$ and we obtain
$$\int_{\mathcal O_\epsilon} |F(u_{2,\delta}(x))|P_\delta(x,x_0)|a|^{-\frac{1}{2}}dx\leq C\left(\sup_{\delta\in[0,1]}|\partial_{\nu_a}u_{1,\delta}(x_0)-\partial_{\nu_a}u_{2,\delta}(x_0)|\right)^{\frac{1}{2}},\quad \delta\in(s_1/R,r_2/R].$$
Therefore, repeating the above process, we get
$$|F(s)|\leq |F(r_2)|\leq C\left(\sup_{\delta\in[0,1]}|\partial_{\nu_a}u_{1,\delta}(x_0)-\partial_{\nu_a}u_{2,\delta}(x_0)|\right)^{\frac{1}{2^2}},\quad s\in[s_1,s_2].$$
By iteration, we deduce \eqref{t3b}.

\textbf{Case 2}: $x_1\neq x_0$. Let us consider $P_M$ the Poisson kernel associated with the equation
$$-\sum_{i,j=1}^n \partial_{x_i}
\left( a_{i,j}(x) \partial_{x_j}v\right)+Mv=0.$$
Fixing $P(x)=P_\delta(x,x_0)-P_M(x,x_0)$, $\delta\in[0,R]$, we deduce that $P\in C^1(\overline{\Omega}\setminus\{x_0\})$ and
$$-\sum_{i,j=1}^n \partial_{x_i}
\left( a_{i,j}(x) \partial_{x_j}P\right)+MP=(M-q_\delta) P_\delta(x,x_0),\quad x\in\Omega.$$
Moreover, \eqref{t3a} implies that $M-q_\delta\geq0$ and we recall that $P_\delta(\cdot,x_0)\geq0$, which implies that $(M-q_\delta) P_\delta(\cdot,x_0)\geq0$. Combining this with the definition of the Poisson kernel and applying the maximum principle, we deduce that
\bel{t3bc}P_\delta(x,x_0)\geq P_M(x,x_0),\quad \delta\in[0,R],\ x\in\Omega.\ee
Recall that the map $x\mapsto P_M(x,x_0)\in C^1(\overline{\Omega}\setminus\{x_0\})$ and, setting $B(x_0,t):=\{x\in\R^n:\ |x-x_0|<t\}$ with $t=|x_1-x_0|/2$, we have
$$-\sum_{i,j=1}^n \partial_{x_i}
\left( a_{i,j}(x) \partial_{x_j} P_M(x,x_0) \right)+MP_M(x,x_0) =0,\quad x\in\Omega\setminus B(x_0,t),$$
$$P_M(x,x_0)=0,\quad x\in\partial\Omega\setminus B(x_0,t).$$
Moreover, \eqref{poisson2}  implies that
\bel{pos}P_M(x,x_0)>0,\quad x\in\Omega\setminus B(x_0,t).\ee
Indeed, assuming the contrary, we may find $x_\star\in\Omega\setminus B(x_0,t)$ such that $P_M(x_\star,x_0)=0$. Then, applying Harnack inequality we can find an open neighborhood  $U$ of $x_\star$  such that $P_M(x,x_0)=0$, $x\in U$, which combined with unique continuation properties for solutions of elliptic equations implies that $P_M(x,x_0)=0$, $x\in\Omega$. This last property contradicts the definition of $P_M$ and it confirms that \eqref{pos} holds true.
Combining \eqref{pos} with the Hopf lemma, we deduce that
$$\partial_\nu P_M(x,x_0)<0,\quad x\in \partial\Omega\setminus B(x_0,t),$$
where $\partial_\nu$ denotes the normal derivative with respect to $\partial\Omega$.
From this property one can check, by eventually using boundary normal coordinates, that there exists a constant $c$ depending on $\Omega$, $a$, $x_0$ and $M$ such that
\bel{poisson4}P_M(x,x_0)\geq c\ \textrm{dist}(x,\partial\Omega),\quad x\in\overline{\Omega}\setminus B(x_0,t).\ee
Combining this estimate with \eqref{t3bc}, and choosing $\epsilon\in(0,t)$ sufficiently small, we obtain
$$\int_{\mathcal O_\epsilon} P_\delta(x,x_0)|a|^{-\frac{1}{2}}dx\geq \int_{\mathcal O_\epsilon} P_M(x,x_0)|a|^{-\frac{1}{2}}dx\geq c \int_{\mathcal O_\epsilon} \textrm{dist}(x,\partial\Omega)dx\geq  c'\epsilon^{n+1},$$
with $c'>0$ depending on $\Omega$, $a$, $x_0$ and $M$. Therefore, repeating the above argumentation, we get
$$|F(\delta R)|\leq C(\epsilon^{-n-1}|\partial_{\nu_a}u_{1,\delta}(x_0)-\partial_{\nu_a}u_{2,\delta}(x_0)|+\epsilon),\quad \delta\in[0,r_1/R]$$
and, choosing $\epsilon=|\partial_{\nu_a}u_{1,\delta}(x_0)-\partial_{\nu_a}u_{2,\delta}(x_0)|^{\frac{1}{n+2}}$ and $s=\delta R$, we obtain
$$|F(s)|\leq C\left(\sup_{\delta\in[0,1]}|\partial_{\nu_a}u_{1,\delta}(x_0)-\partial_{\nu_a}u_{2,\delta}(x_0)|\right)^{\frac{1}{n+2}},\quad s\in[0,r_1].$$
Using this estimate and repeating the arguments of the first case, we get
$$|F(s)|\leq C\left(\sup_{\delta\in[0,1]}|\partial_{\nu_a}u_{1,\delta}(x_0)-\partial_{\nu_a}u_{2,\delta}(x_0)|\right)^{\frac{1}{(n+2)^2}},\quad s\in[0,r_2]$$
and by iteration we obtain \eqref{t3c}.
This completes the proof of Theorem \ref{t3}.
\subsection{Proof of Corollary \ref{c2}}
Since $F_1-F_2$ is of constant sign, without loss of generality we may assume that $F_1-F_2\geq0$. Then, repeating the arguments used for \eqref{t3e} with $\delta=1$ we obtain
$$\begin{aligned}\int_\Omega\abs{F_1(u_{2,1}(x))-F_2(u_{2,1}(x))}P_\delta(x,x_0)|a|^{-\frac{1}{2}}dx&=\int_\Omega F_1(u_{2,1}(x))-F_2(u_{2,1}(x)))P_\delta(x,x_0)|a|^{-\frac{1}{2}}dx\\
&=\partial_{\nu_a}u_{2,1}(x_0)-\partial_{\nu_a}u_{1,1}(x_0)\\
&\leq\abs{\partial_{\nu_a}u_{2,1}(x_0)-\partial_{\nu_a}u_{1,1}(x_0)}.\end{aligned}$$
Combining this estimate with the arguments used in the proof of Theorem \ref{t3}, we obtain \eqref{c2a}-\eqref{c2b}. This completes the proof of Corollary \ref{c2}.

\section{Numerical reconstruction}
\subsection{Computation algorithm }\label{sec5.1}
We will consider an algorithm of reconstruction associated with our inverse problems. More precisely, fixing the set
$$\mathcal B:=\{G|_{[0,R]}:\ G\in C^1(\R),\ G(0)=0,\ G'\leq0\}$$
for any $F\in \mathcal B_0$ and any $\delta\in[0,1]$ we denote by $u_{\delta,F}$ the solution of \eqref{eq1}. Then, fixing the set
$$\mathcal B_\delta:=\{G|_{[0,\delta R]}:\ G\in \mathcal B,\ G'<0\},\quad \delta\in[0,1],$$
$\Gamma$ an open subset of $\partial\Omega$ and the real values $0<\delta_1<\ldots<\delta_N=1$, we will consider the numerical reconstruction of $F\in \mathcal B_1$ from the data $\partial_{\nu_a}u_{\delta_k,F}|_\Gamma$, $k=1,\ldots,N$. In contrast to our theoretical results, for practical reason, we need to extend the measurements at one point to  measurement on the subset $\Gamma$ of $\partial\Omega$. However, in accordance with our theoretical results, the open set $\Gamma$ can chosen as small as possible. The numerical algorithm will be based on Tikonov regularization in a process that will be described bellow.

But first let us consider $d_\delta$ the map defined on $\mathcal B_\delta\times\mathcal B_\delta$ by
$$d_\delta(G,F)=\left(\int_\Omega|G(u_{\delta,G}(x))-F(u_{\delta,F}(x))|^2|a|^{-\frac{1}{2}}dx\right)^{\frac{1}{2}},\quad (G,F)\in \mathcal B_\delta\times\mathcal B_\delta.$$
Recall that by the maximum principle, we have $0\leq u_{\delta,F},u_{\delta,G}\leq\delta R$ and the map $d_\delta$ is suitably defined on $\mathcal B_\delta\times\mathcal B_\delta$. We can also show the following.

\begin{lem}\label{l1} The map $d_\delta$ is a metric on $\mathcal B_\delta\cup\{0\}$.\end{lem}
\begin{proof} We only need to check that for any $F,G\in \mathcal B_\delta\cup\{0\}$ the following implication
$$d_\delta(F,G)=0\Longrightarrow F=G$$
holds true. For this purpose, let us consider $F,G\in \mathcal B_\delta\cup\{0\}$ such that $d_\delta(F,G)=0$. Then, we have
$$F(u_{\delta,F}(x))=G(u_{\delta,G}(x)),\quad x\in\Omega$$
and, sending $x$ to $\partial\Omega$, we get
$$F(\delta g(x))=F(u_{\delta,F}(x))=G(u_{\delta,G}(x))=G(\delta g(x)),\quad x\in\partial\Omega.$$
Combining this with the fact that $g(\partial\Omega)=[0,R]$, we deduce that $F=G$.
\end{proof}

Using this property, we will consider an algorithm of reconstruction based on the minimization of the quantity
$$J(F)=\sum_{k=1}^N\left(\norm{(\partial_{\nu_a}u_{\delta_k,F}-m_k)\chi^{\frac{1}{2}}|a|^{-\frac{1}{4}}}_{L^2(\partial\Omega)}^2+\lambda d_{\delta_k}(F,0)^2\right),$$
with $m_1,\ldots,m_N\in H^{\frac{1}{2}}(\partial\Omega)$ the noisy data associated with the Dirichlet excitations $\delta_1g,\ldots,\delta_Ng$, $\lambda\geq0$ is a regularization parameter and $\chi$ a smooth non-negative cut-off function with supp$(\chi)\subset\Gamma$.

We consider the minimization of the quantity $J(F)$. For this purpose, we will determine first the map $G=F(u_F)$ that we see as a source term. Then, from $G$ we determine the corresponding term by considering the identity
$$F(s)=G(x),\quad x\in\{y\in\Omega:\ u_F(x)=s\}.$$
More precisely, we consider the set $\mathcal B_1$ as an open subset of the Banach space of function $F\in C^1([0,R])$ satisfying $F(0)=0$. Then, we consider the Fr\'echet derivative of $J(F)$ with respect to $F$ at the direction $H$ given by
$$\begin{aligned}DJ(F)H=&2\sum_{k=1}^N\int_{\partial\Omega} \partial_{\nu_a}D_Fu_{\delta_k,F}H(\partial_{\nu_a}u_{\delta_k,F}-m_k)\chi(x) |a|^{-\frac{1}{2}}d\sigma(x)\\
&+2\lambda\sum_{k=1}^N\int_\Omega (H(u_{\delta_k,F}(x))+F'(u_{\delta_k,F}(x))D_Fu_{\delta_k,F}H)F(u_{\delta_k,F}(x))|a|^{-\frac{1}{2}}dx,\end{aligned}$$
where $w_k=D_Fu_{\delta_k,F}H$ solves the boundary value problem
$$\left\{
\begin{array}{ll}
-|a|^{\frac{1}{2}}\sum_{i,j=1}^n \partial_{x_i}
\left( |a|^{-\frac{1}{2}}a_{i,j}(x) \partial_{x_j} w_k \right)-F'(u_{\delta_k,F}(x))w_k=H(u_{\delta_k,F}(x)) & \mbox{in}\ \Omega ,
\\
w_k=0 &\mbox{on}\ \partial\Omega.
\end{array}
\right.$$
For $k=1,\ldots,N$, we fix $z_k$ the solution of the boundary value problem
$$\left\{
\begin{array}{ll}
-|a|^{\frac{1}{2}}\sum_{i,j=1}^n \partial_{x_i}
\left( |a|^{-\frac{1}{2}}a_{i,j}(x) \partial_{x_j} z_k \right)=0 & \mbox{in}\ \Omega ,
\\
z_k=(\partial_{\nu_a}u_{\delta_k,F}-m_k)\chi &\mbox{on}\ \partial\Omega
\end{array}
\right.$$
and integrating by parts we obtain
$$\begin{aligned}&\int_{\partial\Omega} \partial_{\nu_a}D_Fu_{\delta_k,F}H(\partial_{\nu_a}u_{\delta_k,F}-m_k)\chi(x) |a|^{-\frac{1}{2}}d\sigma(x)\\
&=\int_\Omega \sum_{i,j=1}^n \partial_{x_i}
\left( |a|^{-\frac{1}{2}}a_{i,j}(x) \partial_{x_j} D_Fu_{\delta_k,F}H \right)z_kdx\\
&=-\int_\Omega (\left(F'(u_{\delta_k,F}(x))D_Fu_{\delta_k,F}H+H(u_{\delta_k,F}(x))\right)z_k)|a|^{-\frac{1}{2}}dx.\end{aligned}$$
Therefore, we have
$$DJ(F)H=2\sum_{k=1}^N\int_\Omega \left(F'(u_{\delta_k,F}(x))D_Fu_{\delta_k,F}H+H(u_{\delta_k,F}(x))\right)(\lambda F(u_{\delta_k,F}(x))-z_k(x))|a|^{-\frac{1}{2}}dx$$
and the condition
\bel{num1} z_k(x)=\lambda F(u_{\delta_k,F}(x)),\quad k=1,\ldots,N,\ x\in\Omega,\ee
guaranties that $DJ(F)=0$. We will consider a solution $F$ of this nonlinear equation by treating first the determination of $G_k(x)=F(u_{\delta_k,F}(x))$. More precisely, we choose $F_0\in\mathcal B_1$ and we solve \eqref{eq1} for $\delta=\delta_k$, $k=1,\ldots,N$, and $F=F_0$. Then, we fix $G_{k,0}(x)=F_0(u_{\delta_k,F_0}(x))$ and, fixing $M>0$, we consider the following thresholding iteration process (see also \cite{LJY,YLY})
\begin{equation}\label{iter}
G_{k,\ell+1}=\frac{\phi_{k,\ell}}{\lambda+M}+\frac{M G_{k,\ell}}{\lambda+M},
\end{equation}
where $\phi_{k,\ell}$ solves the problem
$$\left\{
\begin{array}{ll}
-|a|^{\frac{1}{2}}\sum_{i,j=1}^n \partial_{x_i}
\left( |a|^{-\frac{1}{2}}a_{i,j}(x) \partial_{x_j} \phi_{k,\ell} \right)=0 & \mbox{in}\ \Omega ,
\\
\phi_{k,\ell}=(\partial_{\nu_a}\psi_{k,\ell}-m_k)\chi &\mbox{on}\ \partial\Omega
\end{array}
\right.$$
with $\psi_{k,\ell}$ the solution of
$$\left\{
\begin{array}{ll}
-|a|^{\frac{1}{2}}\sum_{i,j=1}^n \partial_{x_i}
\left( |a|^{-\frac{1}{2}}a_{i,j}(x) \partial_{x_j} \psi_{k,\ell} \right)=G_{k,\ell} & \mbox{in}\ \Omega ,
\\
\psi_{k,\ell}=\delta_k g &\mbox{on}\ \partial\Omega.
\end{array}
\right.$$

Fixing $\epsilon>0$, we stop this iteration process when the condition
\begin{equation}\label{iterc}
\norm{(G_{k,\ell+1}-G_{k,\ell})|a|^{-\frac{1}{4}}}_{L^2(\Omega)}\leq \epsilon \norm{G_{k,\ell}|a|^{-\frac{1}{4}}}_{L^2(\Omega)}
\end{equation}
is fulfilled. Then, we solve the problem
$$\left\{
\begin{array}{ll}
-|a|^{\frac{1}{2}}\sum_{i,j=1}^n \partial_{x_i}
\left( |a|^{-\frac{1}{2}}a_{i,j}(x) \partial_{x_j} \psi_{k,\ell} \right)=G_{k,\ell} & \mbox{in}\ \Omega ,
\\
\psi_{k,\ell}=\delta_k g &\mbox{on}\ \partial\Omega.
\end{array}
\right.$$
and, fixing $m=0,\ldots,N-1$ and $s\in(\delta_m R,\delta_{m+1}R]$, we consider $F$ to be the function given by
\begin{equation}\label{Fs}
F(s)=\frac{1}{N-m} \sum_{k=m+1}^NG_{k,\ell}(x),\quad x\in\{x\in\Omega:\ \psi_{m+1,\ell}(x)=\ldots=\psi_{N,\ell}(x)=s\},
\end{equation}
where we fix $\delta_0=0$.

\subsection{Numerical experiments}
In this section, we will apply the reconstruction algorithm presented in Section \ref{sec5.1} to the numerical solution of problem \eqref{eq1}. We will present several numerical examples demonstrating the algorithm's performance in determining the unknown function $F$.

To provide a more concrete and illustrative example, we consider a 2D semilinear elliptic equation as the following:
\begin{align}\label{problem}
\left\{
\begin{array}{ll}
-\Delta u_{\delta}(x,y)=F(u_{\delta}),  &{(x,y)\in (0,1)\times(0,1),}\\
u_{\delta}|_{x=0}=\delta y,~u_{\delta}|_{x=1}=\delta y,  &{y\in[0,1]},\\
u_{\delta}|_{y=0}=0,~u_{\delta}|_{y=1}=\delta,  &{x\in[0,1]},\\
\end{array} \right.
\end{align}
In this problem, the domain $\Omega$ is the unit square $(0,1)\times(0,1)$, coefficient
$a=1$, and the boundary conditions are satisfied $g(\partial\Omega)=[0,1]$ with $R=1$.


We use Picard-Successive over relaxation (Picard-SOR) iteration method to solve the forward problem \eqref{problem} as \cite{BJK},
\[\Delta  v^{k+1}=F(v^{k}).\]
As is customary in discretizing the problem for a numerical solution, the Laplacian operator is replaced at the interior grid points of the computational domain by a second-order central difference approximation
\begin{equation}
(v_{m+1,n}+v_{m,n+1}+v_{m-1,n}+v_{m,n-1}-4v_{m,n})/h^2,
\end{equation}
where $v_{m,n}$ represents the function value at the grid point $(m\Delta x,n\Delta y)$ and $\Delta x=\Delta y=h$ are the grid spacing in the $x$ and $y$ directions, respectively.  We set $h$ equal to $1/32$ in inverse problem and $1/64$ in forward problem in order to avoid the "inverse crimes".


We generate the noisy data $m_{k}$ by adding uniform random noises at exact data $\partial_{\nu_a}u_{\delta_k,F}|_\Gamma$
in such a way that
\[m_{k}=\partial_{\nu_a}u_{\delta_k,F}|_\Gamma+\varepsilon_{k}\; rand(-1,1),\]
where $rand(-1,1)$ denotes the random number uniformly distributed in $[-1,1]$. For the noise level $\varepsilon_{k}>0$, we choose it as a certain portion of the amplitude of the noiseless data. Specifically, we define the noise level as:
\[\varepsilon_{k}:=\varepsilon_0~ \max|(\partial_{\nu_a}u_{\delta_k,F}|_\Gamma)|,~0<\varepsilon_0<1,\]
where $\varepsilon_0$ is a parameter that controls the relative noise level added to the data.

 In our numerical examples, the selection of the parameter $M$ in Equation \eqref{iter} as well as the regularization parameters $\lambda$ are determined through repeated trials to achieve the best inversion results, as discussed in the prior some works \cite{LJY,YLY}.
  Furthermore, we fix the iteration stopping tolerance $\epsilon$ at $10^{-3}$ in Equation \eqref{iterc}.
When we reconstruct $F$ using formula \eqref{Fs},
the selected points are chosen as $s_m = \delta_m+\frac{1}{2N}$, $m=0,\cdots,N-1$, i.e. the midpoint of the interval $(\delta_m R,\delta_{m+1}R]$.
The relative $L^2$ error is defined as:
\[err=\frac{\|F-\hat{F}\|_{L^2(\Omega)}}{\|F\|_{L^2(\Omega)}},\]
where $F$ is the true solution and $\hat{F}$ is the reconstructed result obtained using our algorithm.
\begin{example}\label{example1}
In this example, we fix the true solution to be $F=-u^3$. Then evaluate the reconstruction performance by considering different combinations of relative noise levels $\varepsilon_0$ and different initial guesses for our algorithm. We take $N=30$, and $\delta_k=\frac{k}{N}$, $k=1,2,\cdots,N$. For each value of $\delta_k$, we choose the measurement data $\partial_{\nu_a}u_{\delta_k,F}|_{\Gamma_1}$ on the right boundary of the domain $\Omega$, i.e., $\Gamma_1=\{(x,y)\in\Omega |x=1\}$. The specific choices of these parameters as well as the resulting numerical performance are listed in the table \ref{table1}.
\end{example}
\renewcommand{\arraystretch}{1.1}
\begin{table}[!htbp]
\small
\caption{Parameter settings and the corresponding numerical performances in Example \ref{example1}.}
\centering
\begin{tabular}{cccccc}
\toprule[2pt] 
Initial guess & $\varepsilon_0(\%)$ & M & $\lambda$ & err(\%) & Illustration \\
\midrule
\multirow{3}{*}{$F_0=-u$}&  0.5 & 0.8 & $9.2\times 10^{-4}$ & 6.83& \multirow{3}{*}{Figure \ref{fig1:subfig1}}\\
                          & 1 & 0.8 & $9.2\times 10^{-4}$ & 7.35\\
                          & 5 & 0.8 & $9.2\times 10^{-4}$ & 8.82\\
\midrule 
\multirow{3}{*}{$F_0=-u^2$}& 0.5 & 0.8 & $9.4\times 10^{-4}$ & 4.06 & \multirow{3}{*}{Figure \ref{fig1:subfig2}}\\
                            & 1 & 0.8 & $9.3\times 10^{-4}$ & 4.44\\
                            & 5 & 0.8 & $9.1\times 10^{-4}$ & 5.82\\
\bottomrule[2pt] 
\end{tabular}\label{table1}
\end{table}
\begin{figure}[htbp]
\centering
\subfigure[]{
\label{fig1:subfig1}
\includegraphics[width=0.45\textwidth,trim=4 4 35 15,clip]{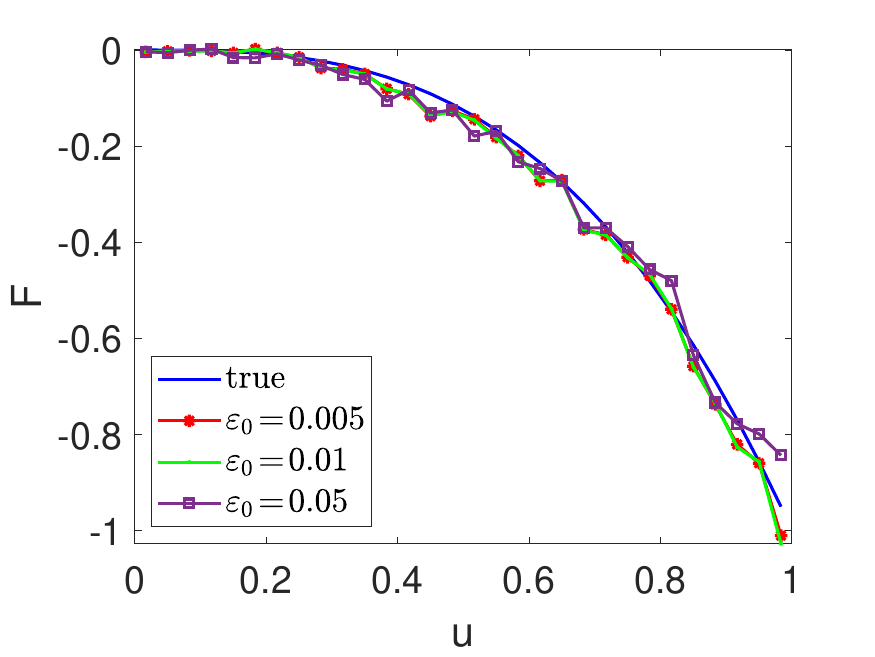}
}
\quad
\subfigure[]{
\label{fig1:subfig2}
\includegraphics[width=0.45\textwidth,trim=4 4 35 15,clip]{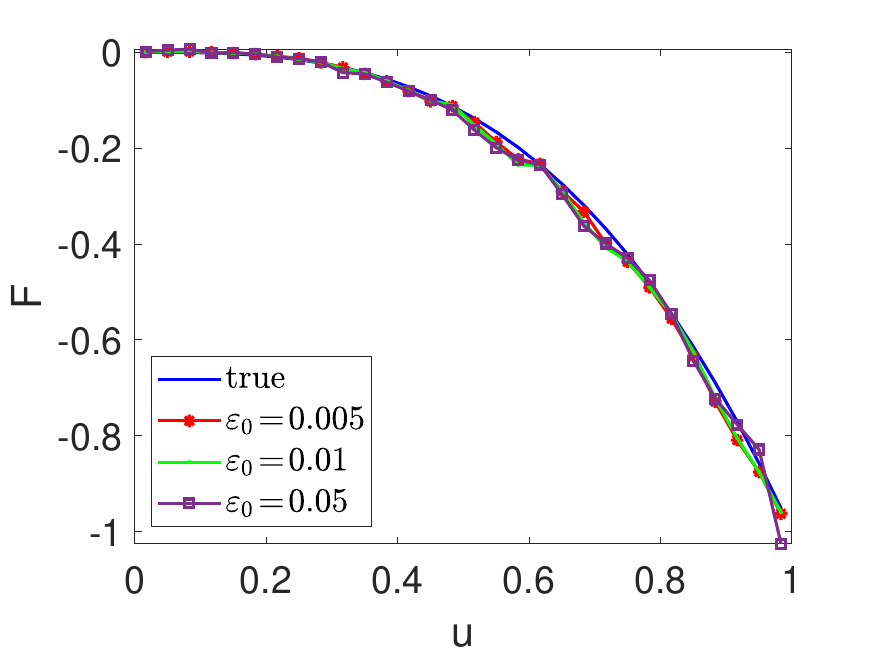}
}
\caption{The true and inversion results of Example \ref{example1}: $(a)~ F_0=-u$, $(b)~ F_0=-u^2$.}
\label{figure1}
\end{figure}
From the analysis of Figure \ref{figure1}, we can see that the performance of our algorithm depends on the choice of the initial guess. The closer the initial guesses are to the true solution $F$, the more accurate the inversion results will be.

Furthermore, our repeated numerical experiments have discovered an important finding only when the initial guess is strictly smaller than the true solution can obtain a reasonable result. This shows the possible existence of local optimal solutions. The algorithm may converge to these solutions with different initial values.
This suggests that the initial guess plays a crucial role in determining the final reconstruction results.

Observed Table \ref{table1}, as the noise level $\varepsilon_0$ increasing from $0.5\%$, $1\%$ to $5\%$, the relative error of the inversion results increase only moderately. This indicates that the algorithm has robustness against the measurement errors.


\begin{example}\label{example2}
In this example, we study the influence of the properties of the true solution on the reconstruction performance. Specifically, we fix the initial guess $F_0 = -u$ and consider different true function $F$:
$(a)~ F^a=-\ln(1+u)$,
$(b)~ F^b=(1-e^u)/2$,
$(c)~ F^c=-\sin (u)$
and $(d)~ F^d=(\cos (\pi u)-1)/2.5$.
By evaluating the reconstruction performance across a range of true function types, we can gain insights into how the algorithm behaves under various
function with different properties.
Furthermore, we assess the algorithm's performance at different relative noise levels $\varepsilon_0$ in the noise data.
In this example, we take $N=30$, and measurement data on $\Gamma_1$ as Example \ref{example1}. The specific parameter choices and the resulting numerical performance are displayed in Table \ref{table2}.
\end{example}
\renewcommand{\arraystretch}{1.1}
\begin{table}[!htbp]
\caption{Parameter settings and the corresponding numerical performances in Example \ref{example2}.}
\centering
\begin{tabular}{cccccc}
\toprule[2pt] 
True solution $F$ & $\varepsilon_0(\%)$ & M & $\lambda$ & err(\%) & Illustration \\
\midrule 
\multirow{3}{*}{$-\ln(1+u)$} & 0.5 & 0.5 & 7.02$\times 10^{-4}$ & 2.2 & \multirow{3}{*}{Figure \ref{fig2:subfig1}} \\
                            & 1 & 0.5 & 7$\times 10^{-4}$ & 2.27\\
                            & 5 & 0.5 & 7$\times 10^{-4}$ & 3.81\\
\midrule
\multirow{3}{*}{$(1-e^u)/2$} & 0.5 & 0.8 & 1.03$\times 10^{-3}$ & 4.03 & \multirow{3}{*}{Figure \ref{fig2:subfig2}} \\
                            & 1 & 0.8 & 1.03$\times 10^{-3}$ & 4.08\\
                            & 5 & 0.8 & 1$\times 10^{-3}$ & 5.15\\

\midrule
\multirow{3}{*}{$-\sin (u)$} & 0.5 & 0.5 & 6.95$\times 10^{-4}$ & 1.74 & \multirow{3}{*}{Figure \ref{fig2:subfig3}} \\
                            & 1 & 0.5 & 6.9$\times 10^{-4}$ & 1.81\\
                            & 5 & 0.5 & 6.4$\times 10^{-4}$ & 3.08\\

\midrule
\multirow{3}{*}{$(\cos (\pi u)-1)/2.5$} & 0.5 & 0.1 & 1.33$\times 10^{-4}$ & 3.73 & \multirow{3}{*}{Figure \ref{fig2:subfig4}} \\
                            & 1 & 0.1 & 1.32$\times 10^{-4}$ & 3.81\\
                            & 5 & 0.1 & 1.28$\times 10^{-4}$ & 5.67\\
\bottomrule[2pt] 
\end{tabular}\label{table2}
\end{table}
\begin{figure}[!htbp]
\centering
\subfigure[]{
\label{fig2:subfig1}
\includegraphics[width=0.45\textwidth,trim=4 4 35 15,clip]{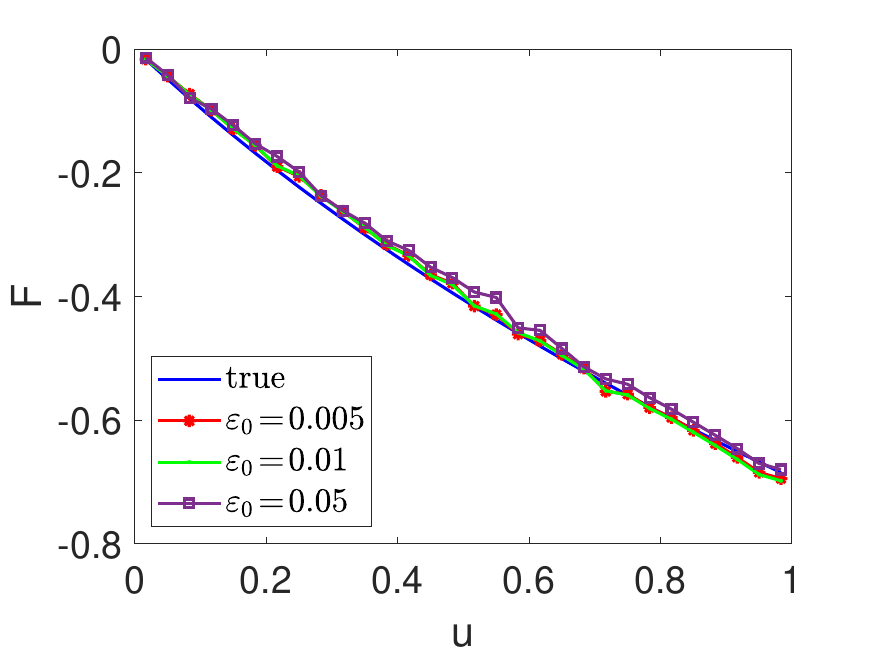}
}
\quad
\subfigure[]{
\label{fig2:subfig2}
\includegraphics[width=0.45\textwidth,trim=4 4 35 15,clip]{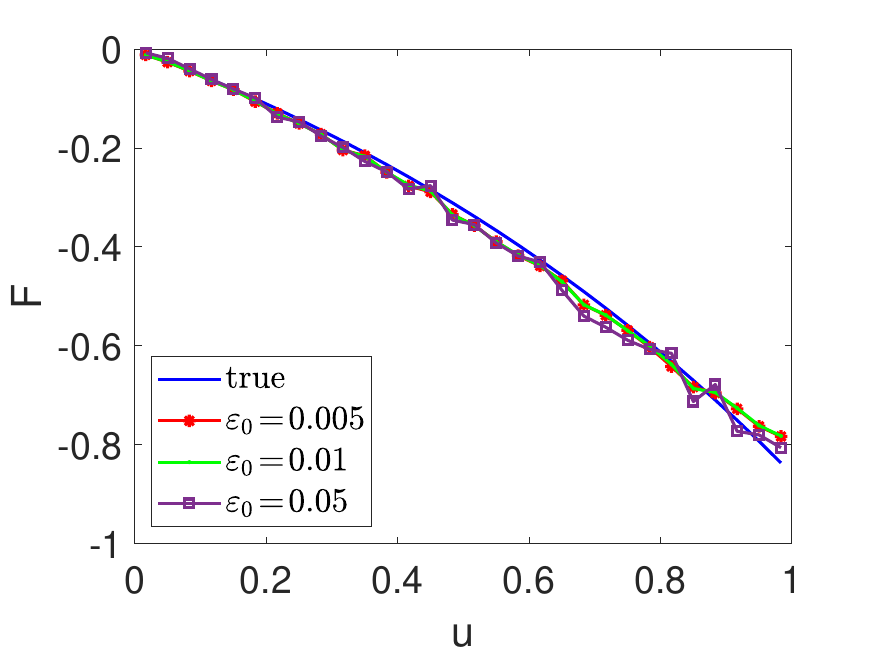}
}
\quad
\subfigure[]{
\label{fig2:subfig3}
\includegraphics[width=0.45\textwidth,trim=4 4 35 15,clip]{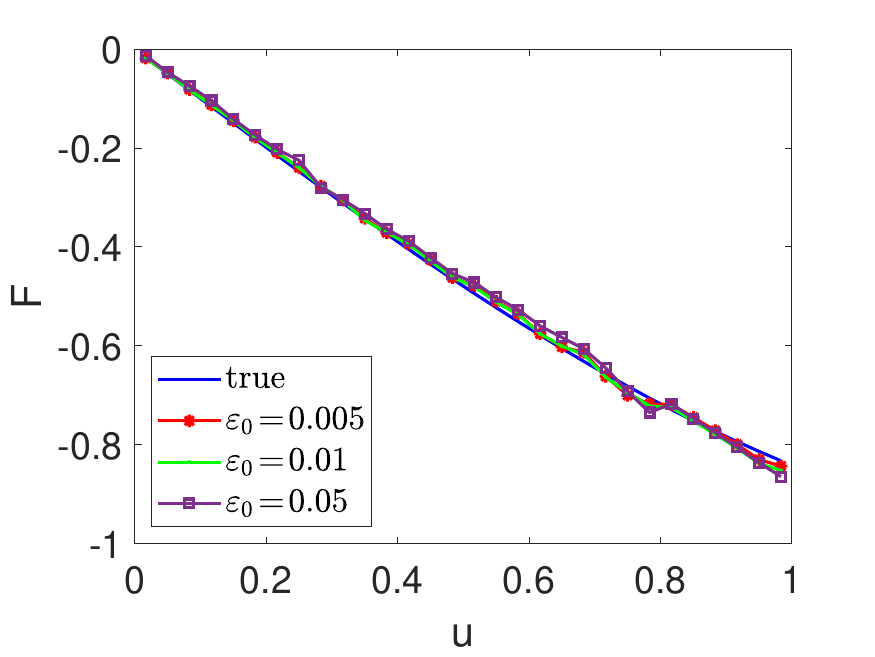}
}
\quad
\subfigure[]{
\label{fig2:subfig4}
\includegraphics[width=0.45\textwidth,trim=4 4 35 15,clip]{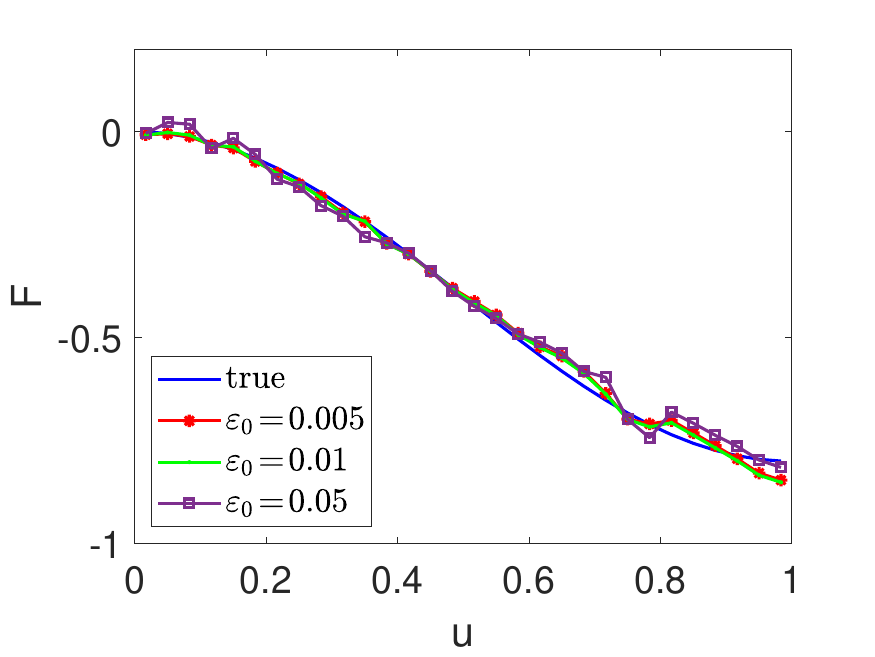}
}
\caption{The true and inversion results of Example \ref{example2}: $(a)~ F^a=-\ln(1+u)$, $(b)~ F^b=(1-e^u)/2$, $(c)~ F^c=-\sin (u)$, $(d)~ F^d=(\cos (\pi u)-1)/2.5$.}
\label{figure2}
\end{figure}
In Example \ref{example2}, we have considered four types of true solution $F$, representing a diverse range of function properties. The results presented in Table \ref{table2} and Figure \ref{figure2} demonstrate that the reconstruction algorithm can yield satisfactory inversion results for these diverse forms of the true solution $F$. This validates the theoretical results and confirms the effectiveness of our algorithm in handling different types of true function forms.

By examining the results shown in Table \ref{table2}, we can also draw a similar conclusion to the previous Example \ref{example1}. As the noise level $\varepsilon_0$ increases from $0.5\%$, $1\%$ to $5\%$, the relative $L^2$ errors in the reconstruction only moderately increase. This indicates that the algorithm has robustness against the measurement errors.

\begin{example}\label{example3}
In this example, we study the impact of the amount of measurement data on the reconstruction results. The true solution is fixed to be $F=-u^3$ and the initial guess to be $F_0=-u^2$. Additionally, the noise level in the observed data is set to $\varepsilon_0=1\%$. We first investigate the effect of shrinking the measurement data subset $\Gamma_1$ to a single point at the midpoint of the right boundary of the domain $\Omega$, i.e., $\Gamma_2=(1,0.5)$.
Next, we fix $\Gamma_1$ and $\Gamma_2$ where the measurement data is collected, and increase the value of $N$ (i.e., the number of $\delta$) from 10 to a larger value of 100.
The specific choices of these parameters as well as the resulting numerical performance are listed in the table \ref{table3}.
\end{example}
From Table \ref{table3} and Figure \ref{figure3}, we can see that the reconstruction algorithm can accurately recover the true solution in all cases with different amount of measurement data. This demonstrates the effectiveness and feasibility of our algorithm. As we expect, the relative errors for measurement subset $\Gamma_1$ are smaller than that for $\Gamma_2$. A larger measurement region can provide more information, thus lead to better inversion results.

For the single-point measurement case ($\Gamma_1$), we find that the relative errors vary slightly for different values of $N$. This indicates that the algorithm is robust to the number of $\delta$ (i.e., $N$) in the case of single-point measurement data. Based on this, we can infer that even with a single $\delta$, the algorithm should still be able to accurately reconstruct the true solution. However, Figure \ref{figure3} shows that the value of $N$ in our algorithm also determines the number of fitting points in the numerical result. When $N$ is too small, the figure of the numerical result may not be able to capture the continuity of the true solution. Therefore, $N$ should not be chosen to be too small.

For measurement subset $\Gamma_1$, there is a situation that is in contrast to the case of measurement subset $\Gamma_2$. Observed Table \ref{table3}, excluding the case of the relatively small $N=10$, we see that as the value of $N$ increases, the relative errors gradually decrease. This suggests that increasing the value of $N$ can improve the inversion results. However, this will lead to an increase in the computational cost. Therefore, we have to choose an appropriate $N$.
\renewcommand{\arraystretch}{1.1}
\begin{table}[!htbp]
\caption{Parameter settings and the corresponding numerical performances in Example \ref{example3}.}
\centering
\begin{tabular}{cccccccc}
\toprule[2pt] 
\multirow{2}{*}{\diagbox{N}{Subset $\Gamma$}}&
\multicolumn{3}{c}{$\Gamma_1$} & \multicolumn{3}{c}{$\Gamma_2$} & \multirow{2}{*}{Illustration}\\
\cmidrule(lr){2-4} \cmidrule(lr){5-7}
    & $M$ & $\lambda$ &  $err(\%)$   & $M$ & $\lambda$ & $err(\%)$\\

\midrule
\multirow{1}{*}{$N=10$}& 0.8 & $9.44\times 10^{-4}$ & 2.67& 0.008 & $8.2\times 10^{-6}$&  6.76 &     Figure \ref{fig3:subfig1}\\


\midrule 
\multirow{1}{*}{$N=20$}& 0.8& $9.46\times 10^{-4}$ & 5.56 & 0.008 & $8.19\times 10^{-6}$ & 7.69 & Figure \ref{fig3:subfig2}\\


\midrule 
\multirow{1}{*}{$N=30$}& 0.8& $9.3\times 10^{-4}$ & 4.43 & 0.008 & $8.19\times 10^{-6}$ & 7.47 & Figure \ref{fig3:subfig3}\\

\midrule 
\multirow{1}{*}{$N=100$}& 0.8& $9.42\times 10^{-4}$ & 3.08 & 0.008 & $8.19\times 10^{-6}$ & 8.13 & Figure \ref{fig3:subfig4}\\

\bottomrule[2pt] 
\end{tabular}\label{table3}
\end{table}
\begin{figure}[!htbp]
\centering
\subfigure[]{
\label{fig3:subfig1}
\includegraphics[width=0.45\textwidth,trim=4 4 35 15,clip]{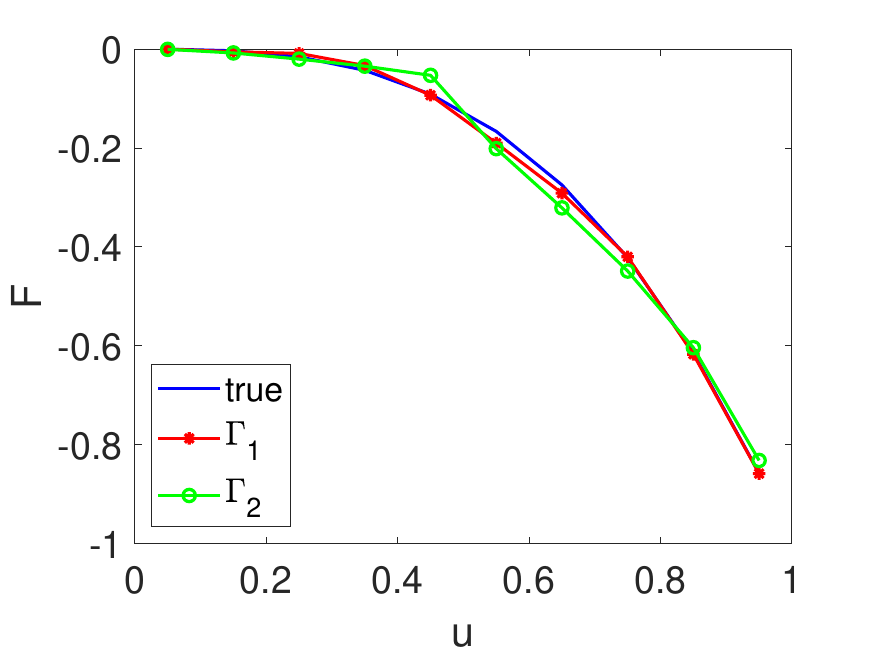}
}
\quad
\subfigure[]{
\label{fig3:subfig2}
\includegraphics[width=0.45\textwidth,trim=5 5 40 15,clip]{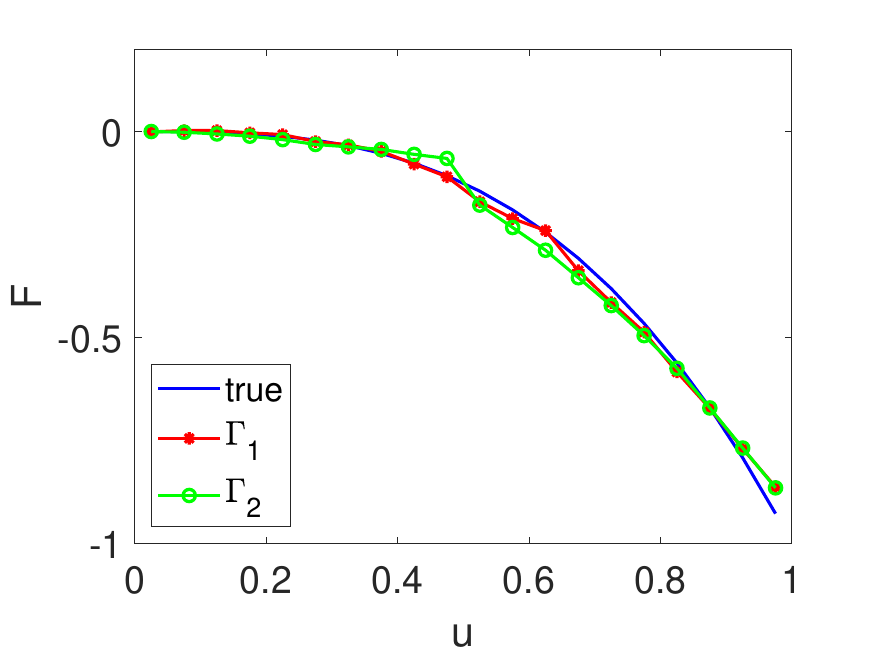}
}
\quad
\subfigure[]{
\label{fig3:subfig3}
\includegraphics[width=0.45\textwidth,trim=4 4 35 15,clip]{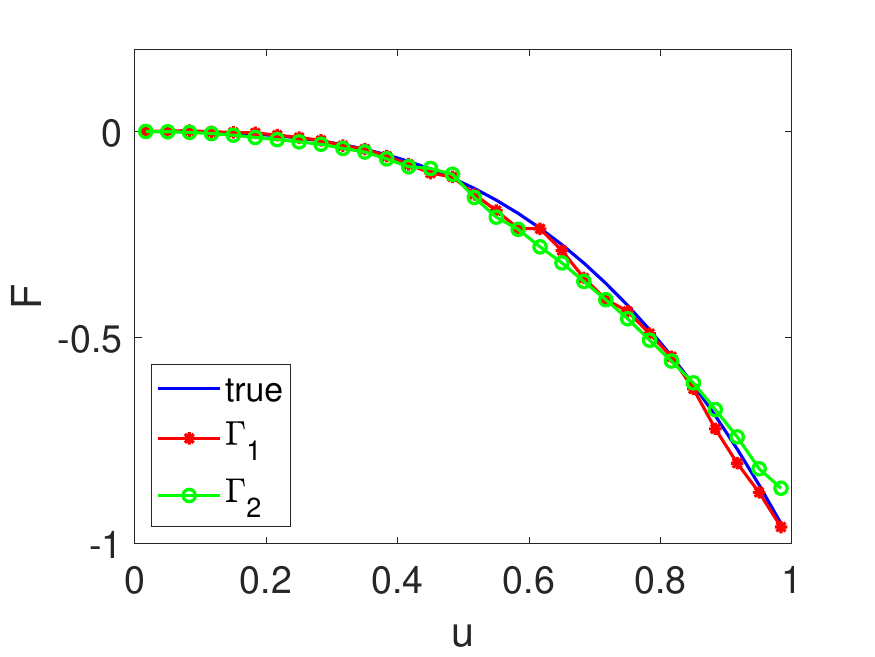}
}
\quad
\subfigure[]{
\label{fig3:subfig4}
\includegraphics[width=0.45\textwidth,trim=4 4 35 15,clip]{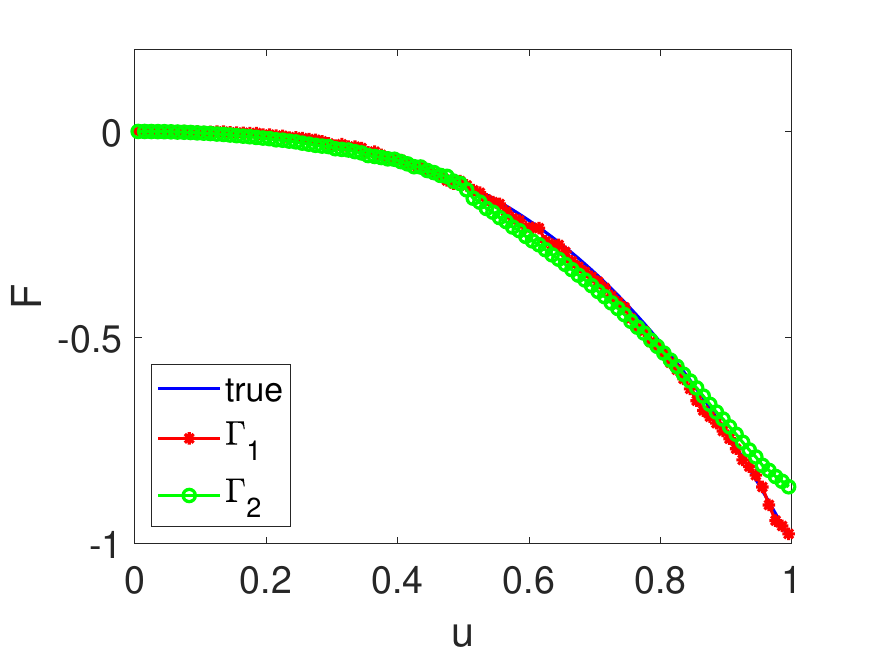}
}
\caption{The true and inversion results of Example \ref{example3}. Legend $\Gamma_1$ and $\Gamma_2$ represent inversion results of the measurement data in subset $\Gamma_1$ and $\Gamma_2$ respectively: $(a)~ N=10$, $(b)~ N=20$, $(c)~ N=30$, $(d)~ N=100$.}
\label{figure3}
\end{figure}


\section*{Acknowledgment}

The work of H. Liu was supported by the NSFC/RGC Joint Research Scheme, N\_CityU101/21; ANR/RGC Joint Research Scheme, A\_CityU203/19; and the Hong Kong RGC General Research Funds (projects 11311122, 12301420 and 11300821).

\bigskip
\vskip 1cm

\end{document}